\documentclass[reqno,10pt]{amsart}
\usepackage{amscd,amssymb,verbatim,array}
\usepackage{hyperref}
\usepackage{amsmath}
\usepackage[active]{srcltx} % SRC Specials: DVI [Inverse] Search

\numberwithin{equation}{section} \theoremstyle{plain}

\newtheorem{theorem}{Theorem}[section]
\newtheorem{lemma}[theorem]{Lemma}

\newtheorem{corollary}[theorem]{Corollary}

\theoremstyle{definition}

\theoremstyle{remark}

\numberwithin{equation}{section}

\newcommand{\Det}{\operatorname{Det}}
\newcommand{\B}{\mathcal{B}}

\newcommand{\even}{\operatorname{even}}
\newcommand{\trivial}{\operatorname{trivial}}
\newcommand{\cyl}{\operatorname{cyl}}
\newcommand{\odd}{\operatorname{odd}}
\newcommand{\gr}{\operatorname{gr}}

\newcommand{\Tan}{\operatorname{tan}}
\newcommand{\Nor}{\operatorname{nor}}
\newcommand{\Dim}{\operatorname{dim}}
\newcommand{\Mod}{\operatorname{mod}}

\newcommand{\rank}{\operatorname{rank}}

\newcommand{\Ker}{\operatorname{ker}}
\newcommand{\ord}{\operatorname{ord}}
\newcommand{\Spec}{\operatorname{Spec}}
\newcommand{\Tr}{\operatorname{Tr}}
\newcommand{\re}{\operatorname{Re}}

\newcommand{\Res}{\operatorname{Res}}
\newcommand{\rres}{\operatorname{res}}
\newcommand{\Imm}{\operatorname{Im}}
\newcommand{\Dom}{\operatorname{Dom}}

\newcommand{\rel}{\operatorname{rel}}
\newcommand{\Abs}{\operatorname{abs}}

\newcommand{\ddet}{\operatorname{det}}
\newcommand{\Id}{\operatorname{Id}}
\newcommand{\supp}{\operatorname{supp}}
\newcommand{\SF}{\operatorname{SF}}
\newcommand{\Mas}{\operatorname{Mas}}

\begin{document}

\title[The gluing formula of the refined analytic torsion]
{The gluing formula of the refined analytic torsion for an acyclic Hermitian connection}
\author{Rung-Tzung Huang}

\address{Institute of Mathematics, Academia Sinica, 6th floor, Astronomy-Mathematics Building, No. 1, Section 4, Roosevelt Road, Taipei, 106-17,Taiwan}

\email{rthuang@math.sinica.edu.tw}

\author{Yoonweon Lee}

\address{Department of Mathematics, Inha University, Incheon, 402-751, Korea}

\email{yoonweon@inha.ac.kr}

\subjclass[2000]{Primary: 58J52; Secondary: 58J28, 58J50}
\keywords{refined analytic torsion, zeta-determinant, eta-invariant, odd signature operator, well-posed boundary condition}
\thanks{The second author was supported by the NRF with the grant number 2010-0008726.}
%\date{\today}

\begin{abstract}
In the previous work ([14]) we introduced the well-posed boundary conditions ${\mathcal P}_{-, {\mathcal L}_{0}}$ and
${\mathcal P}_{+, {\mathcal L}_{1}}$ for the odd signature operator
to define the refined analytic torsion on a compact manifold with boundary.
In this paper we discuss the gluing formula of the refined analytic torsion for an acyclic Hermitian connection
with respect to the boundary conditions ${\mathcal P}_{-, {\mathcal L}_{0}}$ and
${\mathcal P}_{+, {\mathcal L}_{1}}$.
In this case the refined analytic torsion consists of the Ray-Singer analytic torsion, the eta invariant
and the values of the zeta functions at zero.
We first compare the Ray-Singer analytic torsion and eta invariant subject to the boundary condition ${\mathcal P}_{-, {\mathcal L}_{0}}$ or
${\mathcal P}_{+, {\mathcal L}_{1}}$ with the Ray-Singer analytic torsion subject to the relative (or absolute) boundary condition
and eta invariant subject to the APS boundary condition on a compact manifold with boundary.
Using these results together with the well known gluing formula of the Ray-Singer analytic torsion subject to the relative
and absolute boundary conditions
and eta invariant subject to the APS boundary condition, we obtain the main result.
\end{abstract}
\maketitle

\section{Introduction}

\vspace{0.2 cm}

The refined analytic torsion was introduced by M. Braverman and T. Kappeler ([4], [5]) on an odd dimensional closed Riemannian manifold
with a flat bundle as an analytic analogue of the refined combinatorial torsion introduced by M. Farber and V. Turaev ([10], [11], [25], [26]).
Even though these two objects do not coincide exactly, they are closely related.
The refined analytic torsion is defined by using the graded zeta-determinant of the odd signature operator and is described as
an element of the determinant line of the cohomologies.
Specially, when the odd signature operator is defined by an acyclic Hermitian connection on a closed manifold, the refined analytic torsion is a complex number,
whose modulus is the Ray-Singer analytic torsion and the phase part is the $\rho$-invariant determined by the given odd signature operator
and the trivial odd signature operator acting on the trivial line bundle.

In the previous work ([14]) we introduced the well-posed boundary conditions ${\mathcal P}_{-, {\mathcal L}_{0}}$ and ${\mathcal P}_{+, {\mathcal L}_{1}}$
for the odd signature operator, which are complementary to each other and have similar properties as the relative and absolute boundary conditions.
We showed that the refined analytic torsion is well-defined under these boundary conditions on a compact oriented Riemannian manifold with boundary.
In this paper we discuss the gluing formula of the refined analytic torsion with respect to the boundary conditions
${\mathcal P}_{-, {\mathcal L}_{0}}$ and ${\mathcal P}_{+, {\mathcal L}_{1}}$
when the odd signature operator is given by an acyclic Hermitian connection.
In this case the refined analytic torsion consists of the Ray-Singer analytic torsion, the eta invariant
and the values of the zeta functions at zero. The gluing formula of the Ray-Singer analytic torsion with respect to the relative and absolute boundary conditions has been obtained by W. L\"uck ([21]), D. Burghelea, L. Friedlander and T. Kappeler in [9] (cf. [29]).
The gluing formula of the eta invariant
with respect to the Atiyah-Patodi-Singer (APS) boundary condition has
been studied by many authors, for instance, K. Wojciechowski ([32], [33]), U. Bunke ([7]), J. Br\"uning, M. Lesch ([6]), P. Kirk and M. Lesch  ([17]).
To use these results we first compare the Ray-Singer analytic torsion subject to the boundary condition ${\mathcal P}_{-, {\mathcal L}_{0}}$ or
${\mathcal P}_{+, {\mathcal L}_{1}}$ with the Ray-Singer analytic torsion subject to the relative or the absolute boundary condition.
We next compare the eta invariant associated to the odd signature operator subject to ${\mathcal P}_{-, {\mathcal L}_{0}}$ or
${\mathcal P}_{+, {\mathcal L}_{1}}$ with the eta invariant subject to the APS boundary condition.
To compare the Ray-Singer analytic torsions we are going to use the BFK-gluing formula for zeta-determinants ([8], [18], [19]) and the adiabatic limit method.
To compare the eta invariants we are going to follow the method given in [6].
These comparison results together with the well known gluing formulas lead to our main result.
The boundary value problem and the gluing formula of the refined analytic torsion
have been already studied by B. Vertman ([27], [28]) but our method is completely different from what he presented.

\vspace{0.3 cm}

Let $(M, g^{M})$ be a compact oriented odd dimensional Riemannian
manifold with boundary $Y$, where $g^{M}$ is assumed to be a product metric near the boundary $Y$.
We denote the dimension of $M$ by $m = 2r - 1$. Suppose that $\rho : \pi_{1}(M) \rightarrow
GL(n, {\Bbb C})$ is a representation of the fundamental group and $E = {\widetilde M}
\times_{\rho} {\Bbb C}^{n}$ is the associated flat bundle, where ${\widetilde M}$
is a universal covering space of $M$. We choose a flat connection
$\nabla$ and extend it to a covariant differential

$$
\nabla : \Omega^{\bullet}(M, E) \rightarrow \Omega^{\bullet + 1}(M, E).
$$

\noindent
Using the Hodge star operator $\ast_{M}$, we define the involution
$\Gamma = \Gamma(g^{M}) : \Omega^{\bullet}(M, E) \rightarrow \Omega^{m - \bullet}(M, E)$ by

\begin{equation}\label{E:1.1}
\Gamma \omega := i^{r} (-1)^{\frac{q(q+1)}{2}} \ast_{M} \omega, \qquad \omega \in \Omega^{q}(M, E),
\end{equation}

\noindent
where $r$ is given as above by $r = \frac{m+1}{2}$.
It is straightforward to see that $\Gamma^{2} = \Id$.
We define the odd signature operator $\B$ by

\begin{equation}\label{E:1.2}
    \B\ = \  \B(\nabla,g^M) \ := \ \Gamma\,\nabla \ + \ \nabla\,\Gamma:\,\Omega^\bullet(M,E)\ \longrightarrow \  \Omega^\bullet(M,E).
\end{equation}

\noindent
Then $\B$ is an elliptic differential operator of order $1$.
Let $N$ be a collar neighborhood of $Y$ which is isometric to $[0, 1) \times Y$.
Any $q$-form $\omega$ can be written , on $N$, by
$$
\omega = \omega_{\Tan} + du \wedge \omega_{\Nor},
$$
where $\omega_{\Tan}$ and $\omega_{\Nor}$ are the tangential and normal parts of $\omega$
and $du$ is the dual of the inward unit normal vector field $\partial u$ to the boundary $Y$ on $N$.
Then we have a natural isomorphism

\begin{equation}\label{E:1.3}
\Psi : \Omega^{p}(N, E|_{N}) \rightarrow C^{\infty}([0, 1), \Omega^{p}(Y, E|_{Y}) \oplus \Omega^{p-1}(Y, E|_{Y})), \quad
\Psi(\omega_{\Tan} + du \wedge \omega_{\Nor}) = (\omega_{\Tan}, \hspace{0.1 cm} \omega_{\Nor}).
\end{equation}

\vspace{0.2 cm}

\noindent
Using the product structure we can
induce a flat connection $\nabla^{Y} : \Omega^{\bullet}(Y, E|_{Y}) \rightarrow \Omega^{\bullet + 1}(Y, E|_{Y})$ from $\nabla$ and a Hodge
star operator $\ast_{Y} : \Omega^{\bullet}(Y, E|_{Y}) \rightarrow
\Omega^{m-1-\bullet}(Y, E|_{Y})$ from $\ast_{M}$.
We define two maps $\beta$, $\Gamma^{Y}$ by

\vspace{0.2 cm}

\begin{equation}\label{E:1.4}
 \begin{aligned}
  \beta  & :  \Omega^{p}(Y, E|_{Y}) \rightarrow  \Omega^{p}(Y, E|_{Y}), \quad \beta(\omega) = (-1)^{p} \omega \\
  \Gamma^{Y}  & :  \Omega^{p}(Y, E|_{Y}) \rightarrow  \Omega^{m-1-p}(Y,
E|_{Y}),\quad \Gamma^{Y}(\omega) = i^{r-1} (-1)^{\frac{p(p+1)}{2}}
\ast_{Y} \omega.
\end{aligned}
\end{equation}

\vspace{0.2 cm}

\noindent
It is straightforward that

\begin{equation}\label{E:1.5}
\beta^{2} = \Id,  \qquad  \Gamma^{Y} \Gamma^{Y} = \Id.
\end{equation}

\noindent
Simple computation shows that

\begin{equation}\label{E:1.6}
 \Gamma = i \beta \Gamma^{Y} \left( \begin{array}{clcr} 0 & -1 \\ 1 & 0 \end{array} \right), \qquad
\nabla = \left( \begin{array}{clcr} 0 & 0 \\ 1 & 0 \end{array} \right) \nabla_{\partial u} +
\left( \begin{array}{clcr} 1 & 0 \\ 0 & -1 \end{array} \right) \nabla^{Y}.
\end{equation}

\noindent
Hence the odd signature operator $\B$ is expressed, under the isomorphism (\ref{E:1.3}), by

\begin{equation}\label{E:1.8}
 \B = - i \beta \Gamma^{Y} \left\{ \left( \begin{array}{clcr} 1 & 0 \\
0 & 1 \end{array} \right) \nabla_{\partial u} + \left( \begin{array}{clcr}  0 & -1
\\ -1 & 0 \end{array} \right) \left( \nabla^{Y} + \Gamma^{Y} \nabla^{Y} \Gamma^{Y} \right) \right\}.
\end{equation}

\noindent
We denote

\begin{equation}  \label{E:1.9}
{\mathcal \gamma} := - i \beta \Gamma^{Y} \left( \begin{array}{clcr} 1 & 0 \\ 0 & 1 \end{array} \right), \qquad
{\mathcal A} := \left( \begin{array}{clcr}  0 & -1 \\ -1 & 0 \end{array} \right) \left( \nabla^{Y} + \Gamma^{Y} \nabla^{Y} \Gamma^{Y} \right)
\end{equation}

\noindent
so that $\B$ has the form of

\noindent
\begin{equation}  \label{E:1.10}
 \B =  {\mathcal \gamma} \left( \partial_{u} + {\mathcal A} \right) \qquad \text{with} \quad {\mathcal \gamma}^{2} = - \Id, \quad
{\mathcal \gamma} {\mathcal A} = - {\mathcal A} {\mathcal \gamma}.
\end{equation}

\noindent
Since $\nabla_{\partial u} \nabla^{Y} = \nabla^{Y} \nabla_{\partial u}$, we have

\begin{equation}\label{E:1.11}
\B^{2} = - \left( \begin{array}{clcr}  1 & 0 \\ 0 & 1 \end{array} \right) \nabla_{\partial u}^{2} +
\left( \begin{array}{clcr}  1 & 0 \\ 0 & 1 \end{array} \right) \left( \nabla^{Y} + \Gamma^{Y} \nabla^{Y} \Gamma^{Y} \right)^{2} =
\left( - \nabla_{\partial u}^{2} + \B_{Y}^{2} \right) \left( \begin{array}{clcr}  1 & 0 \\ 0 & 1 \end{array} \right),
\end{equation}

\noindent
where
$$
\B_{Y} = \Gamma^{Y} \nabla^{Y} + \nabla^{Y} \Gamma^{Y}.
$$

We next choose a Hermitian inner product $h^{E}$.
All through this paper we assume that $\nabla$ is a Hermitian connection with respect to $h^{E}$,
which means that $\nabla$ is compatible with $h^{E}$, {\it i.e.}
for any $\phi$, $\psi \in C^{\infty}(E)$,

$$
d h^{E}(\phi, \psi) = h^{E}(\nabla \phi, \psi) + h^{E}(\phi, \nabla \psi).
$$

\noindent
The Green formula for $\B$ is given as follows (cf. [14]).

\vspace{0.2 cm}

\begin{lemma} \label{Lemma:1.1}
(1) For $\phi \in \Omega^{q}(M, E)$, $\psi \in \Omega^{m-q}(M,E)$,
$\hspace{0.2 cm}  \langle \Gamma \phi, \hspace{0.1 cm} \psi \rangle_{M} \hspace{0.1 cm}
= \hspace{0.1 cm} \langle \phi, \hspace{0.1 cm} \Gamma \psi \rangle_{M}$.   \newline
(2) For $\phi \in \Omega^{q}(M, E)$, $\psi \in \Omega^{q+1}(M,E)$,
$$
\langle \nabla \phi, \hspace{0.1 cm} \psi \rangle_{M} \hspace{0.1 cm} = \hspace{0.1 cm}
\langle \phi, \hspace{0.1 cm} \Gamma \nabla \Gamma \psi \rangle_{M} \hspace{0.1 cm} - \hspace{0.1 cm}
\langle \phi_{\Tan}|_{Y}, \hspace{0.1 cm} \psi_{\Nor}|_{Y} \rangle_{Y}.
$$
(3) For $\phi$, $\psi \in \Omega^{\even}(M, E)$ or $\Omega^{\odd}(M,E)$,
$$
\langle \B \phi, \hspace{0.1 cm} \psi \rangle_{M} \hspace{0.1 cm} - \hspace{0.1 cm} \langle \phi, \hspace{0.1 cm} \B \psi \rangle_{M}  = \hspace{0.1 cm}
- \langle \phi_{\Tan}|_{Y}, \hspace{0.1 cm} i \beta \Gamma^{Y} (\psi_{\Tan}|_{Y}) \rangle_{Y} \hspace{0.1 cm} -
\langle \phi_{\Nor}|_{Y}, \hspace{0.1 cm} i \beta \Gamma^{Y} (\psi_{\Nor}|_{Y}) \rangle_{Y}
\hspace{0.1 cm} = \hspace{0.1 cm} \langle \phi|_{Y}, \hspace{0.1 cm} ( {\mathcal \gamma} \psi)|_{Y} \rangle_{Y}.
$$
\end{lemma}

\vspace{0.2 cm}

\noindent
{\it Remark} : In the assertions (2) and (3) the signs on the inner products on $Y$ are different from those in [14] because in [14]
$\partial u$ is an outward unit normal vector field.

\vspace{0.2 cm}

We note that $\B_{Y}$ is a self-adjoint elliptic operator on $Y$.
Putting ${\mathcal H}^{\bullet}(Y, E|_{Y}) := \Ker \B^{2}_{Y}$, ${\mathcal H}^{\bullet}(Y, E|_{Y})$ is a finite dimensional vector space and we have
$$
\Omega^{\bullet}(Y, E|_{Y}) = \Imm \nabla^{Y} \oplus \Imm \Gamma^{Y} \nabla^{Y} \Gamma^{Y} \oplus {\mathcal H}^{\bullet}(Y, E|_{Y}).
$$

\noindent
If $\nabla \phi = \Gamma \nabla \Gamma \phi = 0$ for $\phi \in \Omega^{\bullet}(M, E)$, simple computation shows that $\phi$ is expressed,
near the boundary $Y$, by

\begin{equation} \label{E:1.12}
\phi = \nabla^{Y} \phi_{\Tan} + \phi_{\Tan, h} + du \wedge ( \Gamma^{Y} \nabla^{Y} \Gamma^{Y} \phi_{\Nor} + \phi_{\Nor, h}), \qquad
\phi_{\Tan, h}, \hspace{0.1 cm} \phi_{\Nor, h} \in {\mathcal H}^{\bullet}(Y, E|_{Y}).
\end{equation}

\noindent
We define ${\mathcal K}$ by

\begin{equation} \label{E:1.13}
{\mathcal K} := \{ \phi_{\Tan, h} \in {\mathcal H}^{\bullet}(Y, E|_{Y}) \mid \nabla \phi = \Gamma \nabla \Gamma \phi = 0 \},
\end{equation}

\noindent
where $\phi$ has the form (\ref{E:1.12}).
If $\phi$ satisfies $\nabla \phi = \Gamma \nabla \Gamma \phi = 0$, so is $\Gamma \phi$ and hence

\begin{equation} \label{E:2.47}
\Gamma^{Y} {\mathcal K} = \{ \phi_{\Nor, h} \in {\mathcal H}^{\bullet}(Y, E|_{Y}) \mid \nabla \phi = \Gamma \nabla \Gamma \phi = 0 \},
\end{equation}

\noindent
where $\phi$ has the form (\ref{E:1.12}).
The second assertion in Lemma \ref{Lemma:1.1} shows that
${\mathcal K}$ is perpendicular to $\Gamma^{Y} {\mathcal K}$.
We then have the following decomposition (cf. Corollary 8.4 in [17], Lemma 2.4 in [14]).

\begin{equation}  \label{E:1.144}
{\mathcal K} \oplus \Gamma^{Y} {\mathcal K} = {\mathcal H}^{\bullet}(Y, E|_{Y}),
\end{equation}

\noindent
which shows that
$( {\mathcal H}^{\bullet}(Y, E|_{Y}), \hspace{0.1 cm} \langle \hspace{0.1 cm}, \hspace{0.1 cm} \rangle_{Y}, \hspace{0.1 cm} - i \beta \Gamma^{Y} )$
is a symplectic vector space with Lagrangian subspaces ${\mathcal K}$ and $\Gamma^{Y} {\mathcal K}$.
We denote by

\begin{equation}   \label{E:1.1555}
{\mathcal L}_{0} = \left( \begin{array}{clcr} {\mathcal K} \\ {\mathcal K} \end{array} \right), \qquad
{\mathcal L}_{1} = \left( \begin{array}{clcr} \Gamma^{Y} {\mathcal K} \\ \Gamma^{Y} {\mathcal K} \end{array} \right).
\end{equation}

\vspace{0.2 cm}

We next define the orthogonal projections
${\mathcal P}_{-, {\mathcal L}_{0}}$, ${\mathcal P}_{+, {\mathcal L}_{1}} :
\Omega^{\bullet}(Y, E|_{Y}) \oplus \Omega^{\bullet}(Y, E|_{Y}) \rightarrow \Omega^{\bullet}(Y, E|_{Y}) \oplus \Omega^{\bullet}(Y, E|_{Y})$ by
\begin{equation} \label{E:1.112}
\Imm {\mathcal P}_{-, {\mathcal L}_{0}} \hspace{0.1 cm} = \hspace{0.1 cm}
\left( \begin{array}{clcr} \Imm \nabla^{Y} \oplus {\mathcal K} \\ \Imm \nabla^{Y} \oplus {\mathcal K} \end{array} \right), \qquad
\Imm {\mathcal P}_{+, {\mathcal L}_{1}} \hspace{0.1 cm} = \hspace{0.1 cm}
\left( \begin{array}{clcr} \Imm \Gamma^{Y} \nabla^{Y} \Gamma^{Y} \oplus \Gamma^{Y} {\mathcal K} \\
\Imm \Gamma^{Y} \nabla^{Y} \Gamma^{Y} \oplus \Gamma^{Y} {\mathcal K} \end{array} \right) .
\end{equation}

\noindent
Then ${\mathcal P}_{-, {\mathcal L}_{0}}$, ${\mathcal P}_{+, {\mathcal L}_{1}}$ are pseudodifferential operators and give well-posed
boundary conditions for $\B$ and the refined analytic torsion.
We denote by $\B_{{\mathcal P}_{-, {\mathcal L}_{0}}}$ and $\B^{2}_{q, {\mathcal P}_{-, {\mathcal L}_{0}}}$ the realizations of $\B$ and $\B^{2}_{q}$
with respect to ${\mathcal P}_{-, {\mathcal L}_{0}}$, {\it i.e.}

\begin{eqnarray}\label{E:2.17}
\Dom \left( \B_{{\mathcal P}_{-, {\mathcal L}_{0}}} \right) & = &
\left\{ \psi \in \Omega^{\bullet}(M, E) \mid {\mathcal P}_{-, {\mathcal L}_{0}} \left( \psi|_{Y} \right) = 0 \right\},  \nonumber\\
\Dom \left( \B^{2}_{q, {\mathcal P}_{-, {\mathcal L}_{0}}} \right) & = &
\left\{ \psi \in \Omega^{q}(M, E) \mid {\mathcal P}_{-, {\mathcal L}_{0}} \left( \psi|_{Y} \right) = 0, \hspace{0.1 cm}
{\mathcal P}_{-, {\mathcal L}_{0}} \left( (\B \psi)|_{Y} \right) = 0 \right\}.
\end{eqnarray}

\noindent
We define $\B_{{\mathcal P}_{+, {\mathcal L}_{1}}}$, $\B^{2}_{q, {\mathcal P}_{+, {\mathcal L}_{1}}}$, $\B^{2}_{q, \Abs}$, $\B^{2}_{q, \rel}$
and $\B_{\Pi_{>}}$, $\B_{\Pi_{<}}$ (see Section 3) in the similar way.
The following result is straightforward (Lemma 2.11 in [14]).

\vspace{0.2 cm}

\begin{lemma}  \label{Lemma:1.2}
$$
\Ker \B^{2}_{q, {\mathcal P}_{-, {\mathcal L}_{0}}} = \ker \B^{2}_{q, \rel} = H^{q}(M, Y ; E), \qquad
\Ker \B^{2}_{q, {\mathcal P}_{+, {\mathcal L}_{1}}} = \ker \B^{2}_{q, \Abs} = H^{q}(M ; E).
$$
\end{lemma}

\vspace{0.2 cm}

We choose an Agmon angle $\theta$ by $- \frac{\pi}{2} < \theta < 0$.
For ${\frak D} = {\mathcal P}_{-, {\mathcal L}_{0}}$ or ${\mathcal P}_{+, {\mathcal L}_{1}}$ we define the zeta function $\zeta_{\B^{2}_{q, {\frak D}}}(s)$
and eta function $\eta_{\B_{\even, {\frak D}}}(s)$ by

\begin{eqnarray*}
\zeta_{\B^{2}_{q, {\frak D}}}(s) & = & \frac{1}{\Gamma(s)} \int_{0}^{\infty} t^{s-1}
\left( \Tr e^{- t \B^{2}_{q, {\frak D}}} - \Dim \Ker \B^{2}_{q, {\frak D}} \right) dt \\
\eta_{\B_{\even, {\frak D}}}(s) & = & \frac{1}{\Gamma(\frac{s+1}{2})} \int_{0}^{\infty} t^{\frac{s-1}{2}}
\Tr \left( \B e^{- t \B^{2}_{\even, {\frak D}}} \right) dt.
\end{eqnarray*}

\noindent
It was shown in [14] that $\zeta_{\B^{2}_{q, {\frak D}}}(s)$ and $\eta_{\B_{\even, {\frak D}}}(s)$ have regular values at $s=0$.
We define the zeta-determinant and eta-invariant by

\begin{eqnarray}
\log \Det_{2\theta} \B^{2}_{q, {\frak D}} & := & - \zeta_{\B^{2}_{q, {\frak D}}}^{\prime}(0),  \label{E:1.15} \\
\eta(\B_{\even, {\frak D}}) & := & \frac{1}{2} \left( \eta_{\B_{\even, {\frak D}}}(0) + \Dim \Ker \B_{\even, {\frak D}} \right).   \label{E:1.16}
\end{eqnarray}

\vspace{0.2 cm}

\noindent
We denote

\begin{eqnarray}  \label{E:1.188}
& & \Omega^{q}_{-}(M, E) = \Imm \nabla \cap \Omega^{q}(M, E), \qquad
\Omega^{q}_{+}(M, E) = \Imm \Gamma \nabla \Gamma \cap \Omega^{q}(M, E),   \nonumber  \\
& & \Omega^{\even}_{\pm}(M, E) = \sum_{q = \even} \Omega^{q}_{\pm}(M, E),
\end{eqnarray}

\noindent
and denote by $\B_{\even}^{\pm}$ the restriction of $\hspace{0.1 cm} \B_{\even} \hspace{0.1 cm}$ to $\hspace{0.1 cm} \Omega^{\even}_{\pm}(M, E)$.
The graded zeta-determinant $\Det_{\gr, \theta} (\B_{\even, {\frak D}})$ of $\B_{\even}$ with respect to the boundary condition ${\frak D}$
is defined by

$$
\Det_{\gr, \theta} (\B_{\even, {\frak D}}) =  \frac{\Det_{\theta} \B^{+}_{\even, {\frak D}}}
{\Det_{\theta} \left( - \B^{-}_{\even, {\frak D}} \right)} .
$$

\vspace{0.2 cm}

We next define the projections ${\widetilde {\mathcal P}}_{0}$,
${\widetilde {\mathcal P}}_{1} : \Omega^{\bullet}(Y, E|_{Y}) \oplus \Omega^{\bullet}(Y, E|_{Y}) \rightarrow
\Omega^{\bullet}(Y, E|_{Y}) \oplus \Omega^{\bullet}(Y, E|_{Y})$ as follows.
For $\phi \in \Omega^{q}(M, E)$

\vspace{0.2 cm}

$$
{\widetilde {\mathcal P}}_{0} (\phi|_{Y}) = \begin{cases} {\mathcal P}_{-, {\mathcal L}_{0}} (\phi|_{Y}) \quad
\text{if} \quad q \quad \text{is} \quad \text{even} \\
{\mathcal P}_{+, {\mathcal L}_{1}} (\phi|_{Y}) \quad \text{if} \quad q \quad \text{is} \quad \text{odd} ,
\end{cases}
\qquad
{\widetilde {\mathcal P}}_{1} (\phi|_{Y}) = \begin{cases} {\mathcal P}_{+, {\mathcal L}_{1}} (\phi|_{Y}) \quad
\text{if} \quad q \quad \text{is} \quad \text{even} \\
{\mathcal P}_{-, {\mathcal L}_{0}} (\phi|_{Y}) \quad \text{if} \quad q \quad \text{is} \quad \text{odd} .
\end{cases}
$$

\vspace{0.2 cm}

\noindent
We denote by

\begin{equation}  \label{E:1.199}
l_{q} := \Dim \Ker \B_{Y, q}^{2}, \qquad l_{q}^{+} := \Dim {\mathcal K} \cap \Ker \B_{Y, q}^{2}, \qquad \text{and} \qquad
l_{q}^{-} := \Dim \Gamma^{Y} {\mathcal K} \cap  \Ker \B_{Y, q}^{2},
\end{equation}

\vspace{0.2 cm}

\noindent
so that $l_{q} = l_{q}^{+} + l_{q}^{-}$ and $l_{q}^{-} = l_{m-1-q}^{+}$.
Simple computation shows that $\log \Det_{\gr, \theta} (\B_{\even, {\mathcal P}_{-, {\mathcal L}_{0}}})$ and
$\log \Det_{\gr, \theta} (\B_{\even, {\mathcal P}_{+, {\mathcal L}_{1}}})$ are described as follows ([14]).

\begin{eqnarray}
(1) \label{E:1.17}
\hspace{0.2 cm} \log \Det_{\gr, \theta} (\B_{\even, {\mathcal P}_{-, {\mathcal L}_{0}}}) & = &
\frac{1}{2} \sum_{q=0}^{m} (-1)^{q+1} \cdot q \cdot \log \Det_{2\theta} \B^{2}_{q, {\widetilde {\mathcal P}_{0}}}
-  i \pi \hspace{0.1 cm} \eta(\B_{\even, {\mathcal P}_{-, {\mathcal L}_{0}}}) \nonumber \\
& + & \frac{\pi i }{2} \left( \frac{1}{4} \sum_{q=0}^{m-1} \zeta_{\B_{Y, q}^{2}}(0) + \sum_{q=0}^{r-2}(r-1-q) (l_{q}^{+} - l_{q}^{-})  \right). \\
(2) \label{E:1.18}
\hspace{0.2 cm}  \log \Det_{\gr, \theta} (\B_{\even, {\mathcal P}_{+, {\mathcal L}_{1}}})  & = &
\frac{1}{2} \sum_{q=0}^{m} (-1)^{q+1} \cdot q \cdot \log \Det_{2\theta} \B^{2}_{q, {\widetilde {\mathcal P}_{1}}}
- i \pi \hspace{0.1 cm} \eta(\B_{\even, {\mathcal P}_{+, {\mathcal L}_{1}}}) \nonumber \\
& - & \frac{\pi i }{2} \left( \frac{1}{4} \sum_{q=0}^{m-1} \zeta_{\B_{Y, q}^{2}}(0)  +  \sum_{q=0}^{r-2}(r-1-q) (l_{q}^{+} - l_{q}^{-}) \right).
\end{eqnarray}

\vspace{0.2 cm}

To define the refined analytic torsion we introduce the trivial connection $\nabla^{\trivial}$ acting on the trivial bundle $M \times {\Bbb C}$
and define the trivial odd signature operator $\B^{\trivial}_{\even} : \Omega^{\even}(M, {\Bbb C}) \rightarrow \Omega^{\even}(M, {\Bbb C})$ in the same way as
(\ref{E:1.2}). The eta invariant $\eta ( \B^{\trivial}_{\even, {\mathcal P}_{-, {\mathcal L}_{0}}/{\mathcal P}_{+, {\mathcal L}_{1}} } )$
associated to $\B^{\trivial}_{\even}$
and subject to the boundary condition ${\mathcal P}_{-, {\mathcal L}_{0}}/{\mathcal P}_{+, {\mathcal L}_{1}}$  is defined in the same way
as in (\ref{E:1.16}) by simply replacing $\B_{\even, {\mathcal P}_{-, {\mathcal L}_{0}}/{\mathcal P}_{+, {\mathcal L}_{1}} }$
by $\B^{\trivial}_{\even, {\mathcal P}_{-, {\mathcal L}_{0}}/{\mathcal P}_{+, {\mathcal L}_{1}} }$.
When $\nabla$ is acyclic in the de Rham complex,
the refined analytic torsion subject to the boundary condition ${\mathcal P}_{-, {\mathcal L}_{0}}/{\mathcal P}_{+, {\mathcal L}_{1}}$ is defined
by

\begin{eqnarray}
\log T_{{\mathcal P}_{-, {\mathcal L}_{0}}} (g^{M}, \nabla) & = & \log \Det_{\gr, \theta} (\B_{\even, {\mathcal P}_{-, {\mathcal L}_{0}}}) +
\frac{\pi i}{2} (\rank E) \eta_{ \B^{\trivial}_{\even, {\mathcal P}_{-, {\mathcal L}_{0}}} } (0)  \\
\log T_{{\mathcal P}_{+, {\mathcal L}_{1}}} (g^{M}, \nabla) & = & \log \Det_{\gr, \theta} (\B_{\even, {\mathcal P}_{+, {\mathcal L}_{1}}}) +
\frac{\pi i}{2} (\rank E) \eta_{ \B^{\trivial}_{\even, {\mathcal P}_{+, {\mathcal L}_{1}}} } (0)
\end{eqnarray}

\vspace{0.2 cm}

\noindent
The refined analytic torsion on a closed manifold is defined similarly.

In this paper we are going to discuss the gluing formula
of the refined analytic torsion with respect to the boundary conditions ${\mathcal P}_{-, {\mathcal L}_{0}}$ and ${\mathcal P}_{+, {\mathcal L}_{1}}$.
For this purpose in the next two sections we are going to compare the Ray-Singer analytic torsion and eta invariant subject to the boundary
condition ${\mathcal P}_{-, {\mathcal L}_{0}}$ (or ${\mathcal P}_{+, {\mathcal L}_{1}}$) with those subject to the relative and APS boundary conditions,
respectively.

\vspace{0.3 cm}

\section{Comparison of the Ray-Singer analytic torsions}

\vspace{0.2 cm}

In this section we are going to compare the Ray-Singer analytic torsion subject to the boundary condition ${\mathcal P}_{-, {\mathcal L}_{0}}$
with the Ray-Singer analytic torsion subject to the relative boundary condition. For this purpose we are going to use the BFK-gluing formula and
the method of the adiabatic limit for stretching the cylinder part.
We recall that $(M, g^{M})$ is a compact oriented Riemannian manifold with boundary $Y$
with a collar neighbrhood $N = [0, 1) \times Y$ and $g^{M}$ is assumed to be a product metric on $N$.
We denote by $M_{1, 1} = [0, 1] \times Y$ and $M_{2} = M - N$.
To use the adiabatic limit we stretch the cylinder part $M_{1, 1}$ to the cylinder of length $r$ and
denote $M_{1, r} = [0, r] \times Y$ with the product metric and

$$
M_{r} = M_{1, r} \cup_{Y_{r}} M_{2} \quad \text{with} \hspace{0.2 cm} Y_{r} = \{ r \} \times Y.
$$

\noindent
Then we can extend the bundle $E$ and
the odd signature operator $\B$ on $M$ to $M_{r}$ in the natural way and we denote these extensions
by $E_{r}$ and $\B(r)$ ($\B = \B(1)$). We denote the restriction of $\B(r)$ to $M_{1, r}$, $M_{2}$ by $\B_{M_{1, r}}$, $\B_{M_{2}}$.
It is well known (cf. [16], [2]) that the Dirichlet boundary value problem for $\B_{q}^{2}$ on $M_{2}$ has a unique
solution, {\it i.e.} for $f + du \wedge g \in \Omega^{q}(M_{2}, E|_{M_{2}})|_{Y_{r}}$,
there exists a unique $\psi \in \Omega^{q}(M_{2}, E|_{M_{2}})$
such that

$$
\B_{q}^{2}\psi = 0, \qquad \psi|_{Y_{r}} = f + du \wedge g.
$$

\noindent
Let ${\frak D}$ be one of the following boundary conditions :
${\mathcal P}_{-, {\mathcal L}_{0}}$, ${\mathcal P}_{+, {\mathcal L}_{1}}$,
the absolute boundary condition, the relative boundary condition or the Dirichlet boundary condition.
We define the Neumann jump operators

$$
Q_{q, 1, {\frak D}}(r), \hspace{0.1 cm} Q_{q, 2} : \Omega^{q}(Y_{r}, E|_{Y_{r}}) \oplus \Omega^{q-1}(Y_{r}, E|_{Y_{r}}) \rightarrow
\Omega^{q}(Y_{r}, E|_{Y_{r}}) \oplus \Omega^{q-1}(Y_{r}, E|_{Y_{r}})
$$

\vspace{0.2 cm}

\noindent
as follows.
For $f + du \wedge g \in \Omega^{q}(Y_{r}, E|_{Y_{r}}) \oplus \Omega^{q-1}(Y_{r}, E|_{Y_{r}})$, we choose
$\phi \in \Omega^{q}(M_{1, r}, E|_{M_{1, r}})$ and $\psi \in \Omega^{q}(M_{2}, E|_{M_{2}})$ such that

\begin{equation}  \label{E:2.1}
\B_{q, M_{1, r}}^{2}\phi = 0, \qquad \B_{q, M_{2}}^{2} \psi = 0,  \qquad \phi|_{Y_{r}} = \psi|_{Y_{r}} = f + du \wedge g,
\qquad {\frak D}(\phi|_{Y_{0}}) = 0.
\end{equation}

\noindent
Then we define
$$
Q_{q, 1, {\frak D}}(r) (f) \hspace{0.1 cm} = \hspace{0.1 cm} (\nabla_{\partial_{u}} \phi)|_{Y_{r}}, \qquad
Q_{q, 2}(f) \hspace{0.1 cm} = \hspace{0.1 cm} - \hspace{0.1 cm} (\nabla_{\partial_{u}} \psi)|_{Y_{r}},
$$

\noindent
where $\partial u$ is the inward unit normal vector field on $N \subset M$.
We next define the Dirichlet-to-Neumann operator $R_{q, {\frak D}}(r)$ as follows.

$$
R_{q, {\frak D}}(r) : \Omega^{q}(Y_{r}, E|_{Y_{r}}) \oplus \Omega^{q-1}(Y_{r}, E|_{Y_{r}}) \rightarrow
\Omega^{q}(Y_{r}, E|_{Y_{r}}) \oplus \Omega^{q-1}(Y_{r}, E|_{Y_{r}})
$$

\begin{equation} \label{E:2.2}
R_{q, {\frak D}} (r) = Q_{q, 1, {\frak D}}(r) + Q_{q, 2}.
\end{equation}

\vspace{0.1 cm}

\noindent
The following lemma is well known (cf. [18]).

\begin{lemma} \label{Lemma:2.1}
(1) $R_{q, \frak D} (r)$ is a non-negative elliptic pseudodifferential operator of order $1$ and has the form of

\begin{equation} \label{E:2.3}
R_{q, \frak D}(r) \hspace{0.1 cm} = \hspace{0.1 cm}
\left( \begin{array}{clcr} 2 \sqrt{\B_{Y, q}^{2}} & 0 \\ 0 & 2 \sqrt{\B_{Y, q-1}^{2}} \end{array} \right)  \hspace{0.1 cm} +
\hspace{0.2 cm} \text{ a smoothing operator}.
\end{equation}

\noindent
(2) $\ker R_{q, \frak D} = \{ \phi|_{Y_{r}} \mid \phi \in \Ker \B_{q, {\frak D}}^{2} \}$.

\end{lemma}

\vspace{0.2 cm}

We denote by $\B^{2}_{q, M_{1, r}, {\frak D}, D}$ ($\B^{2}_{q, M_{2}, D}$)
the restriction of $\B^{2}_{q}(r)$ to $M_{1, r}$
($M_{2}$) subject to the boundary condition ${\frak D}$ on $Y_{0}$ and the Dirichlet boundary condition on $Y_{r}$
(the Dirichlet boundary condition on $Y_{r}$).
We denote by $\B^{2}_{q, {\frak D}}(r)$ the operator $\B_{q}^{2}(r)$ on $M_{r}$ subject to the boundary condition ${\frak D}$
on $Y_{0}$. Then Lemma \ref{Lemma:1.2} shows that $\Dim \Ker \B^{2}_{q, {\frak D}}(r)$ is a topological invariant.
Let $\Dim \Ker \B^{2}_{q, {\frak D}}(r) = k$ and
$\{\varphi_{1}, \cdots, \varphi_{k} \}$ be an orthonormal basis of $\Ker \B^{2}_{q, {\frak D}}(r)$.
We define a positive definite $k \times k$ Hermitian matrix $A_{q, {\frak D}}(r)$ by
$$
A_{q, {\frak D}}(r) = (a_{ij}), \qquad a_{ij} = \langle \varphi_{i}|_{Y_{0}}, \varphi_{j}|_{Y_{0}} \rangle_{Y_{0}}.
$$
Then the BFK-gluing formula ([8], [18], [19]) is described as follows.
Setting $l_{q} = \Dim \Ker \B_{Y, q}^{2}$, we have

\begin{eqnarray}
\log \Det_{2 \theta} \B^{2}_{q, {\frak D}}(r)
& = & \log \Det_{2 \theta} \B^{2}_{q, M_{1, r}, {\frak D}, D} + \log \Det_{2 \theta} \B^{2}_{q, M_{2}, D}
  + \log \Det_{2 \theta} R_{q, {\frak D}}(r) \nonumber  \\
&  & - \log 2 \cdot (\zeta_{\B^{2}_{Y, q}}(0) + \zeta_{\B^{2}_{Y, q-1}}(0) + l_{q} + l_{q-1})  - \log \ddet A_{q, {\frak D}}(r).
\end{eqnarray}

\vspace{0.2 cm}

\noindent
The above equality can be rewritten as follows.

\begin{corollary}  \label{Corollary:2.3}
\begin{eqnarray*}
(1) & & \log \Det_{2 \theta} \B^{2}_{q, {\mathcal P}_{-, {\mathcal L}_{0}}}(r) \hspace{0.1 cm}  = \hspace{0.1 cm}
 \log \Det_{2 \theta} \B^{2}_{q, M_{1, r}, {\mathcal P}_{-, {\mathcal L}_{0}}, D} + \log \Det_{2 \theta} \B^{2}_{q, M_{2}, D} +
\log \Det_{2 \theta} R_{q, {\mathcal P}_{-, {\mathcal L}_{0}}}(r)  \\
& & \hspace{4.5 cm} - \log 2 \cdot (\zeta_{\B^{2}_{Y, q}}(0) + \zeta_{\B^{2}_{Y, q-1}}(0) + l_{q} + l_{q-1})  - \log \ddet A_{q, {\mathcal P}_{-, {\mathcal L}_{0}}}(r)  \\
(2) & & \log \Det_{2 \theta} \B^{2}_{q, {\mathcal P}_{+, {\mathcal L}_{1}}}(r) \hspace{0.1 cm} = \hspace{0.1 cm}
 \log \Det_{2 \theta} \B^{2}_{q, M_{1, r}, {\mathcal P}_{+, {\mathcal L}_{1}}, D}  + \log \Det_{2 \theta} \B^{2}_{q, M_{2}, D} +
\log \Det_{2 \theta} R_{q, {\mathcal P}_{+, {\mathcal L}_{1}}}(r)  \\
& & \hspace{4.5 cm} - \log 2 \cdot (\zeta_{\B^{2}_{Y, q}}(0) + \zeta_{\B^{2}_{Y, q-1}}(0) + l_{q} + l_{q-1}) - \log \ddet A_{q, {\mathcal P}_{+, {\mathcal L}_{1}}}(r)  \\
(3) & &  \log \Det_{2 \theta} \B^{2}_{q, \rel}(r) \hspace{0.1 cm} = \hspace{0.1 cm}
 \log \Det_{2 \theta} \B^{2}_{q, M_{1, r}, \rel, D} + \log \Det_{2 \theta} \B^{2}_{q, M_{2}, D} +
\log \Det_{2 \theta} R_{q, \rel}(r)  \\
& & \hspace{4.5 cm} - \log 2 \cdot (\zeta_{\B^{2}_{Y, q}}(0) + \zeta_{\B^{2}_{Y, q-1}}(0) + l_{q} + l_{q-1})  - \log \ddet A_{q, \rel}(r) \\
(4)  & & \log \Det_{2 \theta} \B^{2}_{q, \Abs}(r) \hspace{0.1 cm} = \hspace{0.1 cm}
 \log \Det_{2 \theta} \B^{2}_{q, M_{1, r}, \Abs, D} + \log \Det_{2 \theta} \B^{2}_{q, M_{2}, D} +
\log \Det_{2 \theta} R_{q, \Abs}(r)   \\
& & \hspace{4.5 cm} - \log 2 \cdot (\zeta_{\B^{2}_{Y, q}}(0) + \zeta_{\B^{2}_{Y, q-1}}(0) + l_{q} + l_{q-1})  - \log \ddet A_{q, \Abs}(r)
\end{eqnarray*}
\end{corollary}

\vspace{0.2 cm}

\noindent
{\it Remark} : The BFK-gluing formula was proved originally on a closed manifold in [8]. But it can be extended
to a compact manifold with boundary with only minor modification when a cutting hypersurface does not intersect the boundary.

\vspace{0.2 cm}

We define $\Omega^{q}_{\pm}(Y, E|_{Y})$ similarly to (\ref {E:1.188})
and denote $\B_{Y, q}^{2, \pm} := \B^{2}_{Y, q}|_{\Omega^{q}_{\pm}(Y, E|_{Y})}$.
Simple computation leads to the following results.

\vspace{0.2 cm}

\begin{lemma}  \label{Lemma:2.3}
The spectra of $\hspace{0.1 cm} \B^{2}_{q, M_{1, r}, {\mathcal P}_{-, {\mathcal L}_{0}}, D} \hspace{0.1 cm}$,
$\hspace{0.1 cm} \B^{2}_{q, M_{1, r}, {\mathcal P}_{+, {\mathcal L}_{1}}, D} \hspace{0.1 cm}$,
$\hspace{0.1 cm} \B^{2}_{q, M_{1, r}, \rel, D} \hspace{0.1 cm}$ and  $\hspace{0.1 cm} \B^{2}_{q, M_{1, r}, \Abs, D} \hspace{0.1 cm}$
are given as follows. Let $k = 1, 2, 3, \cdots$.

\begin{eqnarray*}
& (1) & Spec \left( \B^{2}_{q, M_{1, r}, {\mathcal P}_{-, {\mathcal L}_{0}}, D} \right)  \hspace{0.1 cm} = \hspace{0.1 cm}
\left\{ \lambda_{q-1, j} + (\frac{k \pi}{r})^{2} \right\} \hspace{0.1 cm} \cup \hspace{0.1 cm} \left\{ \lambda_{q-2, j} + (\frac{k \pi}{r})^{2} \right\}
\hspace{0.1 cm} \cup \hspace{0.1 cm} \left\{ \lambda_{q, j} + (\frac{(k - \frac{1}{2}) \pi}{r})^{2} \right\}   \\
& & \hspace{1.5 cm} \cup \hspace{0.1 cm} \left\{ \lambda_{q-1, j} + (\frac{(k - \frac{1}{2}) \pi}{r})^{2} \right\}
\hspace{0.1 cm} \cup \hspace{0.1 cm} \left\{ (\frac{k \pi}{r})^{2} \right\}_{l_{q}^{+} + l_{q-1}^{+}}
\hspace{0.1 cm} \cup \hspace{0.1 cm} \left\{ (\frac{(k - \frac{1}{2}) \pi}{r})^{2} \right\}_{l_{q}^{-} + l_{q-1}^{-}},  \\
& (2) & Spec \left( \B^{2}_{q, M_{1, r}, {\mathcal P}_{+, {\mathcal L}_{1}}, D} \right)  \hspace{0.1 cm} = \hspace{0.1 cm}
\left\{ \lambda_{q-1, j} + (\frac{(k- \frac{1}{2}) \pi}{r})^{2} \right\} \hspace{0.1 cm} \cup \hspace{0.1 cm}
\left\{ \lambda_{q-2, j} + (\frac{(k - \frac{1}{2}) \pi}{r})^{2} \right\}
\hspace{0.1 cm} \cup \hspace{0.1 cm} \left\{ \lambda_{q, j} + (\frac{k \pi}{r})^{2} \right\}   \\
& & \hspace{1.5 cm}  \cup \hspace{0.1 cm} \left\{ \lambda_{q-1, j} + (\frac{k \pi}{r})^{2} \right\}
\hspace{0.1 cm} \cup \hspace{0.1 cm} \left\{ (\frac{k \pi}{r})^{2} \right\}_{l_{q}^{-} + l_{q-1}^{-}}
\hspace{0.1 cm} \cup \hspace{0.1 cm} \left\{ (\frac{(k - \frac{1}{2}) \pi}{r})^{2} \right\}_{l_{q}^{+} + l_{q-1}^{+}},  \\
& (3) & Spec \left( \B^{2}_{q, M_{1, r}, \rel, D} \right)  \hspace{0.1 cm} = \hspace{0.1 cm}
\left\{ \lambda_{q-1, j} + (\frac{k \pi}{r})^{2} \right\} \hspace{0.1 cm} \cup \hspace{0.1 cm} \left\{ \lambda_{q, j} + (\frac{k \pi}{r})^{2} \right\}
\hspace{0.1 cm} \cup \hspace{0.1 cm} \left\{ \lambda_{q-2, j} + (\frac{(k - \frac{1}{2}) \pi}{r})^{2} \right\}   \\
& & \hspace{1.5 cm}  \cup \hspace{0.1 cm} \left\{ \lambda_{q-1, j} + (\frac{(k - \frac{1}{2}) \pi}{r})^{2} \right\}
\hspace{0.1 cm} \cup \hspace{0.1 cm} \left\{ (\frac{k \pi}{r})^{2} \right\}_{l_{q}}
\hspace{0.1 cm} \cup \hspace{0.1 cm} \left\{ (\frac{(k - \frac{1}{2}) \pi}{r})^{2} \right\}_{l_{q-1}},  \\
& (4) & Spec \left( \B^{2}_{q, M_{1, r}, \Abs, D} \right)  \hspace{0.1 cm} = \hspace{0.1 cm}
\left\{ \lambda_{q-1, j} + (\frac{(k - \frac{1}{2}) \pi}{r})^{2} \right\} \hspace{0.1 cm} \cup \hspace{0.1 cm}
\left\{ \lambda_{q, j} + (\frac{(k - \frac{1}{2}) \pi}{r})^{2} \right\}
\hspace{0.1 cm} \cup \hspace{0.1 cm} \left\{ \lambda_{q-2, j} + (\frac{k \pi}{r})^{2} \right\}   \\
& & \hspace{1.5 cm}  \cup \hspace{0.1 cm} \left\{ \lambda_{q-1, j} + (\frac{k \pi}{r})^{2} \right\}
\hspace{0.1 cm} \cup \hspace{0.1 cm} \left\{ (\frac{k \pi}{r})^{2} \right\}_{l_{q-1}}
\hspace{0.1 cm} \cup \hspace{0.1 cm} \left\{ (\frac{(k - \frac{1}{2}) \pi}{r})^{2} \right\}_{l_{q}},
\end{eqnarray*}
where each $\lambda_{q, j}$ runs on $\Spec \left( \B_{Y, q}^{2, +} \right)$  and
$\left\{ (\frac{k \pi}{r})^{2} \right\}_{l_{q}^{+}}$ means that the multiplicity of each $(\frac{k \pi}{r})^{2}$ is $l_{q}^{+}$.
\end{lemma}

For each $q$ we define

\begin{eqnarray*}
\zeta_{\Delta_{q, N}} (s) & = & \sum_{\lambda_{q, j} \in \Spec(\B_{Y, q}^{2, +})} \sum_{k=1}^{\infty}
\left( \lambda_{q, j} + \left( \frac{(k - \frac{1}{2}) \pi}{r} \right)^{2} \right)^{-s},  \\
\zeta_{\Delta_{q, D}} (s) & = & \sum_{\lambda_{q, j} \in \Spec(\B_{Y, q}^{2, +})} \sum_{k=1}^{\infty}
\left( \lambda_{q, j} + \left( \frac{k \pi}{r} \right)^{2} \right)^{-s}.
\end{eqnarray*}

\noindent
The following result is well known (cf. [22]).

\vspace{0.2 cm}

\begin{lemma}  \label{Lemma:2.4}
We put $\xi_{Y, q}(s) = \frac{\Gamma(s - \frac{1}{2})}{2 \sqrt{\pi} \Gamma(s)} \zeta_{\B_{Y, q}^{2, +}}(s - \frac{1}{2})$. Then :
\begin{eqnarray*}
& (1) & - \zeta_{\Delta_{q, N}}^{\prime} (0) \hspace{0.1 cm} = \hspace{0.1 cm}  - r \xi_{Y, q}^{\prime}(0) + \sum_{\lambda_{q, j} \in \Spec(\B_{Y, q}^{2, +})}
\log ( 1 + e^{- 2 r \sqrt{\lambda_{q, j}}}), \\
& (2) & - \zeta_{\Delta_{q, D}}^{\prime} (0)  \hspace{0.1 cm} = \hspace{0.1 cm}  - r \xi_{Y, q}^{\prime}(0) - \frac{1}{2} \log \Det (\B_{Y, q}^{2, +})
\hspace{0.1 cm} + \sum_{\lambda_{q, j} \in \Spec(\B_{Y, q}^{2, +})}\log ( 1 - e^{- 2 r \sqrt{\lambda_{q, j}}}).
\end{eqnarray*}
\end{lemma}

\begin{proof}
The computation of
$- \zeta_{\Delta_{q, D}}^{\prime} (0)$ was done in Proposition 5.1 of [22].
Using the Poisson summation formula, we have the following identity
$$
\sum_{k=1}^{\infty} e^{- \pi^{2} (k - \frac{1}{2})^{2} t} \hspace{0.1 cm} = \hspace{0.1 cm}
\frac{1}{\sqrt{\pi t}} \hspace{0.1 cm} \left( \frac{1}{2} \hspace{0.1 cm} + \hspace{0.1 cm} 2 \sum_{k=1}^{\infty} e^{-\frac{4 k^{2}}{t}}
 \hspace{0.1 cm} - \hspace{0.1 cm} \sum_{k=1}^{\infty} e^{- \frac{k^{2}}{t}} \right),
$$
from which we can compute $- \zeta_{\Delta_{q, N}}^{\prime} (0)$.
\end{proof}

\vspace{0.2 cm}

\begin{corollary}   \label{Corollary:2.5}
Putting $\hspace{0.1 cm} {\mathcal C}_{q}^{+}(r) = \prod_{\lambda_{q, j} \in \Spec(\B_{Y, q}^{2, +})}
\left( 1 + \frac{2 e^{- r \sqrt{\lambda_{q, j}}}}{e^{r \sqrt{\lambda_{q, j}}} - e^{- r \sqrt{\lambda_{q, j}}}} \right)$, we have
\begin{eqnarray*}
\left( - \zeta_{\Delta_{q, N}}^{\prime} (0) \right) - \left( - \zeta_{\Delta_{q, D}}^{\prime} (0) \right) \hspace{0.1 cm} = \hspace{0.1 cm}
\frac{1}{2} \log \Det \B_{Y, q}^{2, +} \hspace{0.1 cm} + \log {\mathcal C}_{q}^{+}(r).
\end{eqnarray*}
\end{corollary}

\vspace{0.2 cm}

\noindent
If we denote the Riemann zeta function by $\zeta_{R}(s)$, it is well known that
$\zeta_{R}(0) = - \frac{1}{2}$ and $\zeta_{R}^{\prime}(0) = - \frac{1}{2} \log 2 \pi$, which leads to the following result.

\begin{lemma} \label{Lemma:2.6}
Setting
$ \hspace{0.1 cm}
\zeta_{1}(s) = \sum_{k=1}^{\infty} \left( \frac{k \pi}{r} \right)^{-2s}$ and
$\hspace{0.1 cm} \zeta_{2}(s) = \sum_{k=1}^{\infty} \left( \frac{(k - \frac{1}{2}) \pi}{r} \right)^{-2s}$,
we have
$ \hspace{0.1 cm} \zeta_{1}^{\prime}(0) = - \log 2r \hspace{0.1 cm}$ and $\hspace{0.1 cm} \zeta_{2}^{\prime}(0) = - \log 2$.
\end{lemma}

\noindent
Lemma \ref{Lemma:2.3} together with Corollary \ref{Corollary:2.5} and Lemma \ref{Lemma:2.6} yields the following result.

\begin{lemma}  \label{Lemma:2.7}
\begin{eqnarray*}
(1) \hspace{0.2 cm} \log \Det \left( \B^{2}_{q, M_{1, r}, {\mathcal P}_{-, {\mathcal L}_{0}}, D} \right) -
\log \Det \left( \B^{2}_{q, M_{1, r}, \rel, D} \right)  & = &  \frac{1}{2} \left( \log \Det \B_{Y, q}^{2, +} - \log \Det \B_{Y, q-2}^{2, +} \right)  \\
& + & \left( \log {\mathcal C}_{q}^{+}(r) - \log {\mathcal C}_{q-2}^{+}(r) \right) + \left( l_{q-1}^{+} - l_{q}^{-} \right) \log r  \\
(2) \hspace{0.2 cm} \log \Det \left( \B^{2}_{q, M_{1, r}, {\mathcal P}_{+, {\mathcal L}_{1}}, D} \right) -
\log \Det \left( \B^{2}_{q, M_{1, r}, \rel, D} \right)  & = & \left( l_{q-1}^{-} - l_{q}^{+} \right) \log r
\end{eqnarray*}
\begin{eqnarray*}
\noindent
(3) \hspace{0.2 cm}  \sum_{q=0}^{m} (-1)^{q+1} q \left( \log \Det \B^{2}_{q, M_{1, r}, {\widetilde {\mathcal P}}_{0}, D}
- \log \Det \B^{2}_{q, M_{1, r}, \rel, D} \right) & = & \sum_{q=\even} \log \Det \B_{Y, q}^{2, +} \hspace{0.1 cm} + \hspace{0.1 cm}
2 \sum_{q=\even} \log {\mathcal C}_{q}^{+}(r) \\
 + \hspace{0.1 cm} \left( \sum_{q = \even} (2q+1) l_{q}^{-} - \sum_{q = \odd} (2q+1) l_{q}^{+} \right) \log r   & &
\end{eqnarray*}
\begin{eqnarray*}
\noindent
(4) \hspace{0.2 cm}  \sum_{q=0}^{m} (-1)^{q+1} q \left( \log \Det \B^{2}_{q, M_{1, r}, {\widetilde {\mathcal P}}_{1}, D}
- \log \Det \B^{2}_{q, M_{1, r}, \rel, D} \right) & = & - \sum_{q=\odd} \log \Det \B_{Y, q}^{2, +} \hspace{0.1 cm} - \hspace{0.1 cm}
2 \sum_{q=\odd} \log {\mathcal C}_{q}^{+}(r) \\
 + \hspace{0.1 cm} \left( \sum_{q = \even} (2q+1) l_{q}^{+} - \sum_{q = \odd} (2q+1) l_{q}^{-} \right) \log r   & &
\end{eqnarray*}
\begin{eqnarray*}
\noindent
 (5)  \hspace{0.2 cm} \sum_{q=0}^{m} (-1)^{q+1} q \left( \log \Det \B^{2}_{q, M_{1, r}, {\mathcal P}_{-, {\mathcal L}_{0}}, D}
- \log \Det \B^{2}_{q, M_{1, r}, \rel, D} \right) & = &  \frac{m}{2} \hspace{0.1 cm} \chi(Y, E|_{Y}) \hspace{0.1 cm} \log r,
\end{eqnarray*}
where $\hspace{0.1 cm} \chi(Y, E|_{Y}) := \sum_{q=0}^{m-1} (-1)^{q} \cdot l_{q} \hspace{0.1 cm}$ the Euler characteristic of $Y$
with respect to $H^{\bullet}(Y, E|_{Y})$.
\end{lemma}

\vspace{0.2 cm}

We finally discuss the Dirichlet-to-Neumann operator $R_{q, {\frak D}}(r)$ defined by
$R_{q, {\frak D}}(r) = Q_{q, 1, {\frak D}}(r) +  Q_{q, 2}$, where ${\frak D}$ is one of
${\mathcal P}_{-, {\mathcal L}_{0}}$, ${\mathcal P}_{+, {\mathcal L}_{1}}$, the absolute or the relative boundary condition.
The following lemma is straightforward.

\vspace{0.2 cm}

\begin{lemma}  \label{Lemma:2.8}
$R_{q, {\mathcal P}_{-, {\mathcal L}_{0}}}(r)$, $R_{q, {\mathcal P}_{+, {\mathcal L}_{1}}}(r)$ and $R_{q, \rel}(r)$ are described as follows.

\begin{eqnarray*}
R_{q, {\mathcal P}_{-, {\mathcal L}_{0}}}(r) &  = &  Q_{q, 2} \hspace{0.1 cm} + \hspace{0.1 cm}
\left( \begin{array}{clcr} \sqrt{\B_{Y, q}^{2}} & 0 \\ 0 &  \sqrt{\B_{Y, q-1}^{2}}  \end{array}  \right) \hspace{0.1 cm} + \hspace{0.1 cm}
 \begin{cases}
\frac{ 2 \sqrt{\B_{Y}^{2}} e^{- \sqrt{\B_{Y}^{2}} r}}
{e^{\sqrt{\B_{Y}^{2}} r} - e^{- \sqrt{\B_{Y}^{2}} r}}
\hspace{0.2 cm} \text{on} \hspace{0.2 cm}
\Imm  {\mathcal P}_{-, {\mathcal L}_{0}} \cap (\Ker \B_{Y}^{2})^{\perp}\\
\frac{1}{r} \hspace{2.7 cm} \text{on} \hspace{0.2 cm}
\Imm  {\mathcal P}_{-, {\mathcal L}_{0}} \cap \Ker \B_{Y}^{2}\\
 -  \frac{ 2 \sqrt{\B_{Y}^{2}} e^{- \sqrt{\B_{Y}^{2}} r}}
{e^{\sqrt{\B_{Y}^{2}} r} + e^{- \sqrt{\B_{Y}^{2}} r}}
  \hspace{0.2 cm} \text{on} \hspace{0.2 cm}
 \Imm {\mathcal P}_{+, {\mathcal L}_{1}} \cap (\Ker \B_{Y}^{2})^{\perp}\\
0 \hspace{2.8 cm} \text{on} \hspace{0.2 cm}
\Imm  {\mathcal P}_{+, {\mathcal L}_{1}} \cap \Ker \B_{Y}^{2}
\end{cases}  \\
R_{q, {\mathcal P}_{+, {\mathcal L}_{1}}}(r) &  = &  Q_{q, 2} \hspace{0.1 cm} + \hspace{0.1 cm}
\left( \begin{array}{clcr} \sqrt{\B_{Y, q}^{2}} & 0 \\ 0 &  \sqrt{\B_{Y, q-1}^{2}}  \end{array}  \right) \hspace{0.1 cm} + \hspace{0.1 cm}
 \begin{cases}
- \frac{ 2 \sqrt{\B_{Y}^{2}} e^{- \sqrt{\B_{Y}^{2}} r}}
{e^{\sqrt{\B_{Y}^{2}} r} + e^{- \sqrt{\B_{Y}^{2}} r}}
\hspace{0.2 cm} \text{on} \hspace{0.2 cm}
\Imm  {\mathcal P}_{-, {\mathcal L}_{0}} \cap (\Ker \B_{Y}^{2})^{\perp}\\
0 \hspace{2.7 cm} \text{on} \hspace{0.2 cm}
\Imm  {\mathcal P}_{-, {\mathcal L}_{0}} \cap \Ker \B_{Y}^{2}\\
 \frac{ 2 \sqrt{\B_{Y}^{2}} e^{- \sqrt{\B_{Y}^{2}} r}}
{e^{\sqrt{\B_{Y}^{2}} r} - e^{- \sqrt{\B_{Y}^{2}} r}}
  \hspace{0.2 cm} \text{on} \hspace{0.2 cm}
 \Imm {\mathcal P}_{+, {\mathcal L}_{1}} \cap (\Ker \B_{Y}^{2})^{\perp}\\
\frac{1}{r} \hspace{2.8 cm} \text{on} \hspace{0.2 cm}
\Imm  {\mathcal P}_{+, {\mathcal L}_{1}} \cap \Ker \B_{Y}^{2}
\end{cases}  \\
R_{q, \rel}(r) &  = &  Q_{q, 2} \hspace{0.1 cm} + \hspace{0.1 cm}
\left( \begin{array}{clcr} \sqrt{\B_{Y, q}^{2}} & 0 \\ 0 &  \sqrt{\B_{Y, q-1}^{2}}  \end{array}  \right) \hspace{0.1 cm} + \hspace{0.1 cm}
 \begin{cases}
 \frac{ 2 \sqrt{\B_{Y}^{2}} e^{- \sqrt{\B_{Y}^{2}} r}}{e^{\sqrt{\B_{Y}^{2}} r} - e^{- \sqrt{\B_{Y}^{2}} r}}
\hspace{0.2 cm} \text{on} \hspace{0.2 cm} \Omega^{q}_{-}(Y, E|_{Y}) \oplus \Omega^{q}_{+}(Y, E|_{Y})  \\
\frac{1}{r} \hspace{2.7 cm} \text{on} \hspace{0.2 cm} \Ker \B_{Y}^{2} \cap \Omega^{q}(Y, E|_{Y})\\
- \frac{ 2 \sqrt{\B_{Y}^{2}} e^{- \sqrt{\B_{Y}^{2}} r}}{e^{\sqrt{\B_{Y}^{2}} r} + e^{- \sqrt{\B_{Y}^{2}} r}}
  \hspace{0.2 cm} \text{on} \hspace{0.2 cm}  \Omega^{q-1}_{-}(Y, E|_{Y}) \oplus \Omega^{q-1}_{+}(Y, E|_{Y})  \\
0 \hspace{2.8 cm} \text{on} \hspace{0.2 cm}   \Ker \B_{Y}^{2} \cap \Omega^{q-1}(Y, E|_{Y})
\end{cases}
\end{eqnarray*}

\end{lemma}

\vspace{0.2 cm}

We next discuss
$\lim_{r \rightarrow \infty} \left(
\log \Det R_{q, {\mathcal P}_{-, {\mathcal L}_{0}}/{\mathcal P}_{+, {\mathcal L}_{1}}}(r) - \log \Det R_{q, \rel}(r) \right)$
when $H^{q}(M, Y ; E) = \{ 0 \}$ for each $0 \leq q \leq m$.
The Poincar\'e duality and long exact sequence imply that $ H^{q}(M ; E) = H^{q}(Y ; E|_{Y}) = 0$ for each $0 \leq q \leq m$.
Then Lemma \ref{Lemma:1.2} and Lemma \ref{Lemma:2.1} show that $R_{q, {\mathcal P}_{-, {\mathcal L}_{0}}}(r)$,
$R_{q, {\mathcal P}_{+, {\mathcal L}_{1}}}(r)$ and $R_{q, \rel}(r)$ are invertible operators and

\begin{eqnarray*}
\lim_{r \rightarrow \infty} R_{q, {\mathcal P}_{-, {\mathcal L}_{0}}/{\mathcal P}_{+, {\mathcal L}_{1}}}(r) \hspace{0.1 cm} = \hspace{0.1 cm}
\lim_{r \rightarrow \infty} R_{q, \rel}(r) \hspace{0.1 cm} = \hspace{0.1 cm}
Q_{q, 2} \hspace{0.1 cm} + \hspace{0.1 cm}
\left( \begin{array}{clcr} \sqrt{\B_{Y, q}^{2}} & 0 \\ 0 &  \sqrt{\B_{Y, q-1}^{2}}  \end{array}  \right)
\hspace{0.1 cm} = \hspace{0.1 cm} Q_{q, 2} \hspace{0.1 cm} + \hspace{0.1 cm} | {\mathcal A} | .
\end{eqnarray*}

\vspace{0.2 cm}

\noindent
The kernel of $\hspace{0.1 cm} Q_{q, 2} \hspace{0.1 cm} + \hspace{0.1 cm} | {\mathcal A} | \hspace{0.1 cm}$ is described as follows.
For $f \in \Omega^{q}(M_{2}, E)|_{Y}$, choose $\psi \in \Omega^{q}(M_{2}, E)$ such that
$ \B^{2} \psi = 0$ and $\psi|_{Y} = f$.
Then,

\begin{eqnarray}  \label{E:2.55}
0 \hspace{0.1 cm} = \hspace{0.1 cm} \langle \B^{2} \psi, \hspace{0.1 cm} \psi \rangle & = &
\langle \B \psi, \hspace{0.1 cm} \B \psi \rangle \hspace{0.1 cm} + \hspace{0.1 cm}
\langle (\B \psi)|_{Y}, \hspace{0.1 cm} ({\mathcal \gamma} \psi)|_{Y} \rangle_{Y}  \nonumber \\
& = & \langle \B \psi, \hspace{0.1 cm} \B \psi \rangle \hspace{0.1 cm} + \hspace{0.1 cm}
\langle (\nabla_{\partial_{u}} \psi + {\mathcal A} \psi)|_{Y}, \hspace{0.1 cm} f \rangle_{Y}  \nonumber \\
& = & \parallel \B \psi \parallel^{2} \hspace{0.1 cm} - \hspace{0.1 cm}
\langle Q_{q, 2} f, \hspace{0.1 cm} f \rangle_{Y} \hspace{0.1 cm} + \hspace{0.1 cm}
\langle {\mathcal A} f, \hspace{0.1 cm} f \rangle_{Y} ,
\end{eqnarray}

\noindent
which leads to

\begin{eqnarray}  \label{E:2.56}
\langle (Q_{q, 2} + |{\mathcal A}|)f, \hspace{0.1 cm} f \rangle_{Y} & = &
\parallel \B \psi \parallel^{2} \hspace{0.1 cm} + \hspace{0.1 cm}
\langle (|{\mathcal A}| + {\mathcal A})f, \hspace{0.1 cm} f \rangle_{Y}.
\end{eqnarray}

\noindent
Hence, $f \in \Ker (Q_{q, 2} + |{\mathcal A}|)$ if and only if $\B \psi = 0$ and $(|{\mathcal A}| + {\mathcal A})f = 0$.
From the assumption $H^{\bullet}(Y, E|_{Y}) = 0$ ${\mathcal A}$ is an invertible operator,
which shows that $\psi$ is expressed, on a collar neighborhood of $Y$, by

\begin{equation}  \label{E:2.77}
\psi = \sum_{\lambda_{j} \in \Spec ({\mathcal A}) \atop \lambda_{j} < 0} a_{j} e^{- \lambda_{j} u} \phi_{j}, \qquad \text{where}
\quad {\mathcal A} \phi_{j} = \lambda_{j} \phi_{j}.
\end{equation}

Let $M_{\infty} := \left( (- \infty, 0] \times Y \right) \cup_{Y} M_{2}$.
We can extend $E$ and $\B$ canonically to $M_{\infty}$, which we denote by
$E_{\infty}$ and $\B_{\infty}$. Then $\psi$ in (\ref{E:2.77}) can be extended to $M_{\infty}$ as an $L^{2}$-solution of $\B_{\infty}$.
Hence,
$$
\Ker (Q_{q, 2} + |{\mathcal A}|) = \{ \psi|_{Y} \mid \psi \hspace{0.2 cm} \text{is} \hspace{0.2 cm} \text{an} \hspace{0.2 cm}
L^{2}\text{-solution} \hspace{0.2 cm} \text{of} \hspace{0.2 cm} \B_{\infty} \hspace{0.2 cm} \text{in}
\hspace{0.2 cm} \Omega^{q}(M_{\infty}, E_{\infty}) \hspace{0.1 cm} \}.
$$

\vspace{0.2 cm}

\noindent
It is a well known fact (Proposition 4.9 in [1]) that the space of $L^{2}$-solutions of $\B_{\infty}$ is isomorphic to the image of
$H^{\bullet}(M, Y ; E) \rightarrow H^{\bullet}(M ; E)$, which is zero under our assumption. This shows that $(Q_{q, 2} + |{\mathcal A}|)$
is injective and hence invertible, which leads to the following result.

\vspace{0.2 cm}

\begin{lemma} \label{Lemma:2.99}
We assume that for each $0 \leq q \leq m$, $H^{q}(M ; E) = H^{q}(M, Y ; E) = \{ 0 \}$.
Then
\begin{eqnarray*}
\lim_{r \rightarrow \infty} \log \Det R_{q, {\mathcal P}_{-, {\mathcal L}_{0}}/{\mathcal P}_{+, {\mathcal L}_{1}}}(r)
\hspace{0.1 cm} = \hspace{0.1 cm}
\log \Det \lim_{r \rightarrow \infty} R_{q, \rel}(r) \hspace{0.1 cm} = \hspace{0.1 cm}
\log \Det \left( Q_{q, 2} \hspace{0.1 cm} + \hspace{0.1 cm} | {\mathcal A}| \right) .
\end{eqnarray*}
\end{lemma}

\vspace{0.2 cm}

\noindent
Corollary \ref{Corollary:2.3} and Lemma \ref{Lemma:2.7} together with Lemma \ref{Lemma:2.99} lead to the following result.

\vspace{0.2 cm}

\begin{corollary} \label{Corollary:2.10}
We assume that for each $0 \leq q \leq m$, $H^{q}(M ; E) = H^{q}(M, Y ; E) = \{ 0 \}$.
Then :
\begin{eqnarray*}
& (1) & \lim_{r \rightarrow \infty}
\sum_{q=0}^{m} (-1)^{q+1} q \cdot \left( \log \Det_{2 \theta} \B^{2}_{q, {\widetilde {\mathcal P}}_{0}}(r) -
\log \Det_{2 \theta} \B^{2}_{q, \rel}(r) \right) = \frac{1}{4} \sum_{q = 0}^{m-1} \log \Det_{2 \theta} \B_{Y, q}^{2}.  \\
& (2) & \lim_{r \rightarrow \infty}
\sum_{q=0}^{m} (-1)^{q+1} q \cdot \left( \log \Det_{2 \theta} \B^{2}_{q, {\widetilde {\mathcal P}}_{1}}(r) -
\log \Det_{2 \theta} \B^{2}_{q, \rel}(r) \right) = - \frac{1}{4} \sum_{q = 0}^{m-1} \log \Det_{2 \theta} \B_{Y, q}^{2}.  \\
& (3) & \lim_{r \rightarrow \infty}
\sum_{q=0}^{m} (-1)^{q+1} q \cdot \left( \log \Det_{2 \theta} \B^{2}_{q, {\mathcal P}_{-, {\mathcal L}_{0}}}(r) -
\log \Det_{2 \theta} \B^{2}_{q, \rel}(r) \right) \\
& & \hspace{1.0 cm} = \hspace{0.1 cm} \lim_{r \rightarrow \infty}
\sum_{q=0}^{m} (-1)^{q+1} q \cdot \left( \log \Det_{2 \theta} \B^{2}_{q, {\mathcal P}_{+, {\mathcal L}_{1}}}(r) -
\log \Det_{2 \theta} \B^{2}_{q, \Abs}(r) \right) \hspace{0.1 cm} = \hspace{0.1 cm} 0 .
\end{eqnarray*}
\end{corollary}

\vspace{0.2 cm}

The following lemma is well known (cf. [4], [20]).

\begin{lemma}  \label{Lemma:2.11}
Let $M$ be a compact manifold with boundary $Y$ and $N$ be a collar neighborhood of $Y$.
We suppose that $\{ g_{t}^{M} \mid - \delta_{0} < t < \delta_{0} \}$ is a family of metrics such that each $g^{M}_{t}$ is a product metric and
does not vary on $N$.
Let ${\frak D}$ be one of the following boundary conditions :
${\widetilde {\mathcal P}}_{0}$, ${\widetilde {\mathcal P}}_{1}$, the absolute or the relative boundary condition.
We denote by $\B^{2}_{q, {\frak D}}(t)$ the square of the odd signature operator acting on $q$-forms subject to ${\frak D}$ with respect to the metric $g^{M}_{t}$.
If $H^{q}(M ; E) = H^{q}(M, Y ; E) = \{ 0 \}$ for each $0 \leq q \leq m$, then we have
$$
\frac{d}{dt} \left( \sum_{q=0}^{m} (-1)^{q+1} \cdot q \cdot \log \Det \B^{2}_{q, {\frak D}}(t) \right) \hspace{0.1 cm} = \hspace{0.1 cm} 0.
$$
\end{lemma}

\vspace{0.3 cm}

We fix $\delta_{0} > 0$ sufficiently small and
choose a smooth function $\hspace{0.1 cm} f(r, u) : [0, \infty) \times [0, 1] \rightarrow [0, \infty), \hspace{0.1 cm}$ ($r \geq 1$) such that
$$
\supp_{u}f(r, u) \subset [\delta_{0}, 1 - \delta_{0}], \quad \int_{0}^{1} f(r, u) du = r - 1 , \quad
\text{and} \quad f(1, u) \equiv 0.
$$
Setting $F(r, u) = u + \int_{0}^{u} f(r, t) dt$,
$\quad F_{r} := F(r, \cdot) : [0, 1] \rightarrow [0, r] \hspace{0.1 cm}$ is a diffeomorphism satisfying

\begin{eqnarray*}
F_{r}(u) = \begin{cases} u  \hspace{1.2 cm} \text{for} \hspace{0.2 cm}  0 \leq u \leq \delta_{0}  \\
 u + r -1  \hspace{0.6 cm} \text{for} \hspace{0.2 cm}  1 - \delta_{0} \leq u \leq 1.
\end{cases}
\end{eqnarray*}

\noindent
Let $g^{M}_{r}$ be a metric on $M_{r} : = \left( [0, r] \times Y \right) \cup_{\{ r \} \times Y} M_{2}$
which is a product one on $[0, r] \times Y$.
Then $F_{r}^{\ast}g^{M}_{r}$ is a metric on $M$, which is
$\left( \begin{array}{clcr} F^{\prime}(u)^{2} & 0 \\ 0 & g_{Y} \end{array} \right)$ on $[0, 1] \times Y$.
Hence, $F_{r}^{\ast}g^{M}_{r}$ is a metric on $M$ which is a product one near $Y$. Furthermore,
$(M, F_{r}^{\ast}g^{M}_{r})$ and $(M_{r}, g^{M}_{r})$ are isometric.
Let ${\tilde \B}(r)$ and $\B(r)$ be the odd signature operators defined on $M$ and $M_{r}$ associated to the metrics $F_{r}^{\ast}g^{M}_{r}$
and $g^{M}_{r}$, respectively.
We now assume that for each $0 \leq q \leq m$, $H^{q}(M ; E) = H^{q}(M, Y ; E) = \{ 0 \}$.
Then ${\widetilde \B}^{2}_{q, {\frak D}}(r)$ and $\B^{2}_{q, {\frak D}}(r)$ are invertible operators.
Lemma \ref{Lemma:2.11} leads to the following equalities.

\begin{eqnarray*}
& & \sum_{q=0}^{m} (-1)^{q+1} \cdot q \cdot
\left( \log \Det_{2 \theta} \B_{q, {\widetilde {\mathcal P}}_{0}}^{2} - \log \Det_{2 \theta} \B_{q, \rel}^{2} \right)  \\
& = & \sum_{q=0}^{m} (-1)^{q+1} \cdot q \cdot
\left( \log \Det_{2 \theta} {\tilde \B(r)^{2}_{q, {\widetilde {\mathcal P}}_{0}}} -
\log \Det_{2 \theta} {\tilde \B(r)^{2}_{q, \rel}} \right)  \\
& = & \sum_{q=0}^{m} (-1)^{q+1} \cdot q \cdot
\left( \log \Det_{2 \theta} \B^{2}_{q, {\widetilde {\mathcal P}}_{0}}(r) - \log \Det_{2 \theta} \B^{2}_{q, \rel}(r) \right)  \\
& = & \lim_{r \rightarrow \infty} \sum_{q=0}^{m} (-1)^{q+1} \cdot q \cdot
\left( \log \Det_{2 \theta} \B^{2}_{q, {\widetilde {\mathcal P}}_{0}}(r) -
\log \Det_{2 \theta} \B^{2}_{q, \rel}(r) \right)
\hspace{0.1 cm} = \hspace{0.1 cm}  \frac{1}{4} \sum_{q = 0}^{m-1} \log \Det_{2 \theta} \B_{Y, q}^{2}.
\end{eqnarray*}

\noindent
Similarly, we have
$$
\sum_{q=0}^{m} (-1)^{q+1} \cdot q \cdot
\left( \log \Det_{2 \theta} \B_{q, {\widetilde {\mathcal P}}_{1}}^{2} - \log \Det_{2 \theta} \B_{q, \rel}^{2} \right)
\hspace{0.1 cm} = \hspace{0.1 cm}  - \frac{1}{4} \sum_{q = 0}^{m-1} \log \Det_{2 \theta} \B_{Y, q}^{2}.
$$

\noindent
Corollary \ref{Corollary:2.3}, Corollary \ref{Corollary:2.10}, the Poincar\'e duality and the above equality
lead to the following theorem, which is the main result of this section.

\vspace{0.2 cm}

\begin{theorem}  \label{Theorem:2.11}
Let $(M, g^{M})$ be a compact Riemannian manifold with boundary $Y$ and $g^{M}$ be a product metric near $Y$.
We assume that for each $0 \leq q \leq m$, $H^{q}(M ; E) = H^{q}(M, Y ; E) = \{ 0 \}$. Then :

\begin{eqnarray*}
& (1) &  \sum_{q=0}^{m} (-1)^{q+1} \cdot q \cdot \log \Det_{2 \theta} \B_{q, {\widetilde {\mathcal P}}_{0}}^{2}  \hspace{0.1 cm} = \hspace{0.1 cm}
\sum_{q=0}^{m} (-1)^{q+1} \cdot q \cdot \log \Det_{2 \theta} \B_{q, \rel}^{2}
\hspace{0.1 cm} + \hspace{0.1 cm} \frac{1}{4} \sum_{q = 0}^{m-1} \log \Det_{2 \theta} \B_{Y, q}^{2}  \\
& (2) &  \sum_{q=0}^{m} (-1)^{q+1} \cdot q \cdot \log \Det_{2 \theta} \B_{q, {\widetilde {\mathcal P}}_{1}}^{2}  \hspace{0.1 cm} = \hspace{0.1 cm}
\sum_{q=0}^{m} (-1)^{q+1} \cdot q \cdot \log \Det_{2 \theta} \B_{q, \rel}^{2}
\hspace{0.1 cm} - \hspace{0.1 cm} \frac{1}{4} \sum_{q = 0}^{m-1} \log \Det_{2 \theta} \B_{Y, q}^{2}  \\
& (3) & \sum_{q=0}^{m} (-1)^{q+1} \cdot q \cdot \log \Det_{2 \theta} \B_{q, {\mathcal P}_{-, {\mathcal L}_{0}}}^{2}   \hspace{0.1 cm} = \hspace{0.1 cm}
\sum_{q=0}^{m} (-1)^{q+1} \cdot q \cdot \log \Det_{2 \theta} \B_{q, {\mathcal P}_{+, {\mathcal L}_{1}}}^{2}  \\
&  & \hspace{0.2 cm} = \hspace{0.1 cm} \sum_{q=0}^{m} (-1)^{q+1} \cdot q \cdot \log \Det_{2 \theta} \B_{q, \rel}^{2}
\hspace{0.1 cm} = \hspace{0.1 cm} \sum_{q=0}^{m} (-1)^{q+1} \cdot q \cdot \log \Det_{2 \theta} \B_{q, \Abs}^{2}
\end{eqnarray*}
\end{theorem}

\vspace{0.3 cm}

\section{Comparison of the eta invariants}

\vspace{0.2 cm}

In this section we are going to compare the eta-invariant $\eta(\B_{\even, {\mathcal P}_{-, {\mathcal L}_{0}}})$
with $\eta(\B_{\even, \Pi_{>, {\mathcal L}_{0}}})$, the eta-invariant of $\B_{\even}$ subject to ${\mathcal P}_{-, {\mathcal L}_{0}}$
and the generalized APS boundary condition $\Pi_{>, {\mathcal L}_{0}}$,
where $\Pi_{>} : \Omega^{\even}(M, E)|_{Y} \rightarrow \Omega^{\even}(M, E)|_{Y}$ is the orthogonal projection onto the space spanned by
the positive eigenspaces of ${\mathcal A}$ (cf. (\ref{E:1.9})).
For this purpose we are going to follow the arguments in [6] strongly.
Throughout this section we write the odd signature operator acting on $\Omega^{\even}(M, E)$ by $\B$
rather than $\B_{\even}$ for simplicity.
We begin with the descriptions of $\Imm \Pi_{>}$ and $\Imm {\mathcal P}_{-}$ as graphs of some unitary operators.

We denote by $\left( \Omega^{\even}(M, E)|_{Y} \right)^{\ast}$ the orthogonal complement of
$\left( \begin{array}{clcr} {\mathcal H}^{\even}(Y, E|_{Y}) \\ {\mathcal H}^{\odd}(Y, E|_{Y}) \end{array} \right)$ in
$\left( \Omega^{\even}(M, E)|_{Y} \right)$.
Then the action of the unitary operator ${\mathcal \gamma}$ splits according to the following decomposition.

$$
{\mathcal \gamma} :
\left( \Omega^{\even}(M, E)|_{Y} \right)^{\ast} \oplus
\left( \begin{array}{clcr} {\mathcal H}^{\even}(Y, E|_{Y}) \\ {\mathcal H}^{\odd}(Y, E|_{Y}) \end{array} \right)
\hspace{0.1 cm} \rightarrow \hspace{0.1 cm} \left( \Omega^{\even}(M, E)|_{Y} \right)^{\ast} \oplus
\left( \begin{array}{clcr} {\mathcal H}^{\even}(Y, E|_{Y}) \\ {\mathcal H}^{\odd}(Y, E|_{Y}) \end{array} \right)
$$

\vspace{0.2 cm}

\noindent
Since $\hspace{0.1 cm} {\mathcal \gamma}^{2} = - \Id \hspace{0.1 cm}$, we denote
the $\pm i$-eigenspace of ${\mathcal \gamma}$ in  $\left( \Omega^{\even}(M, E)|_{Y} \right)^{\ast}$ by
$\hspace{0.1 cm} \left( \Omega^{\even}(M, E)|_{Y} \right)^{\ast}_{\pm i} \hspace{0.1 cm}$, which are

\begin{eqnarray}  \label{E:3.1}
\left( \Omega^{\even}(M, E)|_{Y} \right)^{\ast}_{\pm i} & = & \frac{I \mp i {\mathcal \gamma}}{2}
\left( \Omega^{\even}(M, E)|_{Y} \right)^{\ast} .
\end{eqnarray}

\vspace{0.2 cm}

\noindent
It is a well known fact that $\Imm \Pi_{>}$ and $\Imm {\mathcal P}_{-}$ are expressed by the graphs of some unitary operators from
$\left( \Omega^{\even}(M, E)|_{Y} \right)^{\ast}_{+i}$ to $\left( \Omega^{\even}(M, E)|_{Y} \right)^{\ast}_{-i}$.
When restricted to $\left( \Omega^{\even}(M, E)|_{Y} \right)^{\ast}$, $B_{Y}^{2}$ is an invertible operator and we denote its inverse by
$\left( B_{Y}^{2} \right)^{-1}$.
In view of (\ref{E:1.3}) we define $U_{\Pi_{>}}$, $U_{{\mathcal P}_{-}}$ as follows.

\begin{equation} \label{E:3.51}
U_{\Pi_{>}}, \hspace{0.2 cm} U_{{\mathcal P}_{-}}  : \left( \Omega^{\even}(M, E)|_{Y} \right)^{\ast}_{+i}
 \rightarrow  \left( \Omega^{\even}(M, E)|_{Y} \right)^{\ast}_{-i}
\end{equation}

\begin{eqnarray}  \label{E:3.52}
U_{\Pi_{>}} & = & (\B_{Y}^{2})^{-\frac{1}{2}} \left( \nabla^{Y} + \Gamma^{Y} \nabla^{Y} \Gamma^{Y} \right) \left( \begin{array}{clcr} 0 & -1 \\
-1 & 0 \end{array} \right)  \nonumber  \\
U_{{\mathcal P}_{-}} & = & (\B_{Y}^{2})^{- 1} \left( (\B_{Y}^{2})^{-} - (\B_{Y}^{2})^{+} \right) \left( \begin{array}{clcr}  1 & 0 \\
0 & 1 \end{array} \right) ,
\end{eqnarray}

\vspace{0.2 cm}

\noindent
where
$(\B_{Y}^{2})^{-} := \nabla^{Y} \Gamma^{Y} \nabla^{Y} \Gamma^{Y} : \Omega^{\bullet}_{-}(Y, E|_{Y}) \rightarrow \Omega^{\bullet}_{-}(Y, E|_{Y})$ and
$(\B_{Y}^{2})^{+} := \Gamma^{Y} \nabla^{Y} \Gamma^{Y} \nabla^{Y} : \Omega^{\bullet}_{+}(Y, E|_{Y}) \rightarrow \Omega^{\bullet}_{+}(Y, E|_{Y})$.
Then $U_{\Pi_{>}}$ and $U_{{\mathcal P}_{-}}$ are well defined $\Psi$DO's and their adjoints are given by

\begin{equation} \label{E:3.53}
U_{\Pi_{>}}^{\ast}, \hspace{0.2 cm} U_{{\mathcal P}_{-}}^{\ast}  : \left( \Omega^{\even}(M, E)|_{Y} \right)^{\ast}_{-i}
 \rightarrow  \left( \Omega^{\even}(M, E)|_{Y} \right)^{\ast}_{+i}
\end{equation}

\begin{eqnarray}  \label{E:3.54}
U_{\Pi_{>}}^{\ast} & = & (\B_{Y}^{2})^{-\frac{1}{2}} \left( \nabla^{Y} + \Gamma^{Y} \nabla^{Y} \Gamma^{Y} \right) \left( \begin{array}{clcr} 0 & -1 \\
-1 & 0 \end{array} \right) \nonumber  \\
U_{{\mathcal P}_{-}}^{\ast} & = & (\B_{Y}^{2})^{- 1} \left( (\B_{Y}^{2})^{-} - (\B_{Y}^{2})^{+} \right) \left( \begin{array}{clcr}  1 & 0 \\
0 & 1 \end{array} \right) .
\end{eqnarray}

\vspace{0.2 cm}

\noindent
The following lemma is straightforward.

\vspace{0.2 cm}

\begin{lemma}  \label{Lemma:3.1}
(1) Both $U_{\Pi_{>}}$ and $U_{{\mathcal P}_{-}}$  are unitary operators satisfying
$$
U_{\Pi_{>}}^{\ast} \hspace{0.1 cm} U_{\Pi_{>}} \hspace{0.1 cm} = \hspace{0.1 cm} U_{{\mathcal P}_{-}}^{\ast} \hspace{0.1 cm}  U_{{\mathcal P}_{-}}
\hspace{0.1 cm} = \hspace{0.1 cm} \Id, \qquad
{\mathcal \gamma} \hspace{0.1 cm} U_{\Pi_{>}} = - \hspace{0.1 cm}  U_{\Pi_{>}} \hspace{0.1 cm}  {\mathcal \gamma}, \qquad
{\mathcal \gamma} \hspace{0.1 cm}  U_{{\mathcal P}_{-}} = - \hspace{0.1 cm}  U_{{\mathcal P}_{-}} \hspace{0.1 cm}  {\mathcal \gamma}.
$$
\noindent
(2) $\hspace{0.2 cm} \Imm \Pi_{>}$ $(\hspace{0.1 cm} \Imm \Pi_{<} \hspace{0.1 cm})$ and $\hspace{0.1 cm} \Imm {\mathcal P}_{-}$
$(\hspace{0.1 cm} \Imm {\mathcal P}_{+} \hspace{0.1 cm})$ are graphs of
$\hspace{0.1 cm} U_{\Pi_{>}}$ $(\hspace{0.1 cm} - \hspace{0.1 cm} U_{\Pi_{>}} \hspace{0.1 cm})$ and
$\hspace{0.1 cm} U_{{\mathcal P}_{-}}$ $( \hspace{0.1 cm} - \hspace{0.1 cm} U_{{\mathcal P}_{-}} \hspace{0.1 cm})$,
respectively, {\it i.e.}
\begin{eqnarray*}
\Imm \Pi_{>} = \{ x + U_{\Pi_{>}} x \mid x \in \left( \Omega^{\even}(M, E)|_{Y} \right)^{\ast}_{+ i} \}, & &
\Imm \Pi_{<} = \{ x - U_{\Pi_{>}} x \mid x \in \left( \Omega^{\even}(M, E)|_{Y} \right)^{\ast}_{+ i} \},  \\
\Imm {\mathcal P}_{-} = \{ x + U_{{\mathcal P}_{-}} x \mid x \in \left( \Omega^{\even}(M, E)|_{Y} \right)^{\ast}_{+ i} \}, & &
\Imm {\mathcal P}_{+} = \{ x - U_{{\mathcal P}_{-}} x \mid x \in \left( \Omega^{\even}(M, E)|_{Y} \right)^{\ast}_{+ i} \}.
\end{eqnarray*}
(3) $\hspace{0.2 cm} U_{\Pi_{>}}$ anticommutes with $U_{{\mathcal P}_{-}}$ in the following sense, {\it i.e.}
$$ U_{\Pi_{>}}^{\ast} \hspace{0.1 cm}  U_{{\mathcal P}_{-}}   =  - \hspace{0.1 cm}   U_{{\mathcal P}_{-}}^{\ast}  \hspace{0.1 cm} U_{\Pi_{>}}, \qquad
 U_{\Pi_{>}} \hspace{0.1 cm}  U_{{\mathcal P}_{-}}^{\ast}   =  - \hspace{0.1 cm}   U_{{\mathcal P}_{-}}  \hspace{0.1 cm} U_{\Pi_{>}}^{\ast}.
$$
\end{lemma}

\vspace{0.4 cm}

We define $P(\theta) : \left( \Omega^{\even}(M, E)|_{Y} \right)^{\ast}_{i} \rightarrow  \left( \Omega^{\even}(M, E)|_{Y} \right)^{\ast}_{-i}$ by

\begin{equation}
P(\theta) :=  U_{\Pi_{>}} cos \theta + U_{{\mathcal P}_{-}} sin \theta, \qquad (0 \leq \theta \leq \frac{\pi}{2}).
\end{equation}

\vspace{0.2 cm}

\noindent
Then $P(\theta)$ is a unitary operator satisfying the property (1) in Lemma \ref{Lemma:3.1} and a smooth path connecting $U_{\Pi_{>}}$ and
$U_{{\mathcal P}_{-}}$.
We here note that the orthogonal projections $\hspace{0.1 cm} \Pi_{>}$, ${\mathcal P}_{-} :
\left( \Omega^{\even}(M, E)|_{Y} \right)^{\ast}_{+i} \oplus \left( \Omega^{\even}(M, E)|_{Y} \right)^{\ast}_{-i}
 \rightarrow \left( \Omega^{\even}(M, E)|_{Y} \right)^{\ast}_{+i} \oplus \left( \Omega^{\even}(M, E)|_{Y} \right)^{\ast}_{-i} \hspace{0.1 cm}$
 are expressed as follows.

\vspace{0.2 cm}

$$
\Pi_{>} = \frac{1}{2} \left( \begin{array}{clcr} \Id & U_{\Pi_{>}}^{\ast} \\ U_{\Pi_{>}} & \Id \end{array} \right) {\mathcal P}_{\ast} , \qquad
{\mathcal P}_{-} = \frac{1}{2} \left( \begin{array}{clcr} \Id &  U_{{\mathcal P}_{-}}^{\ast} \\ U_{{\mathcal P}_{-}} & \Id
\end{array} \right) {\mathcal P}_{\ast} ,
$$

\vspace{0.2 cm}

\noindent
where ${\mathcal P}_{\ast}$ is the orthogonal projection onto $\left( \Omega^{\even}(M, E)|_{Y} \right)^{\ast}$.
Let ${\mathcal L}_{0} = \left( \begin{array}{clcr} {\mathcal K} \\ {\mathcal K} \end{array} \right) \cap \hspace{0.1 cm} \left( \Omega^{\even}(M, E)|_{Y} \right)$
and ${\mathcal L}_{1} = \left( \begin{array}{clcr} \Gamma^{Y} {\mathcal K} \\
\Gamma^{Y} {\mathcal K} \end{array} \right) \cap \hspace{0.1 cm} \left( \Omega^{\even}(M, E)|_{Y} \right)$
so that ${\mathcal L}_{0} \oplus {\mathcal L}_{1} = \left( \begin{array}{clcr} {\mathcal H}^{\even}(Y, E|_{Y}) \\ {\mathcal H}^{\odd}(Y, E|_{Y})
\end{array} \right)$.
We denote by ${\mathcal P}_{{\mathcal L}_{0}}$ and ${\mathcal P}_{{\mathcal L}_{1}}$ the orthogonal projections onto ${\mathcal L}_{0}$ and ${\mathcal L}_{1}$.
We define the orthogonal projections ${\mathcal P}_{-, {\mathcal L}_{0}}$ and $\Pi_{>, {\mathcal L}_{0}}$  on $\Omega^{\even}(M, E)|_{Y}$ as follows.

\vspace{0.2 cm}

\begin{equation}  \label{E:3.77}
{\mathcal P}_{-, {\mathcal L}_{0}} := {\mathcal P}_{-} + {\mathcal P}_{{\mathcal L}_{0}}, \qquad
\Pi_{>, {\mathcal L}_{0}} := \Pi_{>} + {\mathcal P}_{{\mathcal L}_{0}}
\end{equation}

\vspace{0.2 cm}

\noindent
We define ${\mathcal P}_{+, {\mathcal L}_{1}}$ and $\Pi_{<, {\mathcal L}_{1}}$ in the same way.
Similarly, we define the orthogonal projection ${\widetilde P}(\theta)$ by

\vspace{0.2 cm}

\begin{equation}
{\widetilde P}(\theta) := \frac{1}{2} \left( \begin{array}{clcr} \Id & P(\theta)^{\ast} \\ P(\theta) & \Id \end{array} \right) {\mathcal P}_{\ast}
+ {\mathcal P}_{{\mathcal L}_{0}}
= \Pi_{>} cos \theta + {\mathcal P}_{-} sin \theta + \frac{1}{2} ( 1 - cos \theta - sin \theta ) \hspace{0.1 cm} {\mathcal P}_{\ast}
+ {\mathcal P}_{{\mathcal L}_{0}}.
\end{equation}

\vspace{0.2 cm}

\noindent
${\widetilde P}(\theta)$ satisfies the following properties.

\vspace{0.2 cm}

\begin{lemma} \label{Lemma:3.2}
(1) ${\mathcal \gamma} \hspace{0.1 cm} {\widetilde P}(\theta) \hspace{0.1 cm} = \hspace{0.1 cm} (I - {\widetilde P}(\theta))
\hspace{0.1 cm} {\mathcal \gamma}, \hspace{0.2 cm}$ and
 $\hspace{0.2 cm} {\widetilde P}(\theta) \hspace{0.1 cm} \B_{Y}^{2} \hspace{0.1 cm} = \hspace{0.1 cm} \B_{Y}^{2} {\widetilde P}(\theta)$.  \newline
(2) ${\widetilde P}(\theta) \hspace{0.1 cm} {\mathcal A} \hspace{0.1 cm} {\widetilde P}(\theta)
\hspace{0.1 cm} = \hspace{0.1 cm} cos \theta \hspace{0.1 cm} \vert {\mathcal A} \vert \hspace{0.1 cm} {\widetilde P}(\theta)
\hspace{0.1 cm} = \hspace{0.1 cm} cos \theta \hspace{0.1 cm} \sqrt{(\B_{Y}^{2})} \hspace{0.1 cm} {\widetilde P}(\theta)$.
\end{lemma}

\vspace{0.2 cm}
\begin{proof} : The proofs are straightforward. For the second statement we may need the following identities.
$$
\Pi_{>} {\mathcal P}_{-} \Pi_{>} \hspace{0.1 cm} = \hspace{0.1 cm} \frac{1}{2} \Pi_{>}, \qquad
 {\mathcal P}_{-} \Pi_{>} \hspace{0.1 cm} + \hspace{0.1 cm} \Pi_{>}  {\mathcal P}_{-} \hspace{0.1 cm} = \hspace{0.1 cm}
 \left(  {\mathcal P}_{-} \hspace{0.1 cm} + \hspace{0.1 cm} \Pi_{>} \hspace{0.1 cm} - \hspace{0.1 cm} \frac{1}{2} \Id \right) {\mathcal P}_{\ast}.
$$
\end{proof}

\begin{lemma} \label{Lemma:3.33}
Let $\B_{{\widetilde P}(\theta)}$ be the realization of $\B$ with respect to ${\widetilde P}(\theta)$, {\it i.e.}  \newline
$\Dom \left( \B_{{\widetilde P}(\theta)} \right) = \{ \phi \in H^{1} \left( \Omega^{\even}(M, E) \right) \mid  {\widetilde P}(\theta) (\phi|_{Y}) = 0 \}$.
Then $\B_{{\widetilde P}(\theta)}$ is essentially self-adjoint.
\end{lemma}

\noindent
\begin{proof} :
It was shown in [24] (cf. [12]) that the adjoint $\left( \B_{{\widetilde P}(\theta)} \right)^{\ast}$ is the realization of $\B^{\ast} = \B$ with
respect to the boundary condition $\left( I - {\widetilde P}(\theta) \right) \gamma^{\ast}$, {\it i.e.}

$$
\Dom \left( \B_{{\widetilde P}(\theta)} \right)^{\ast} \hspace{0.1 cm} = \hspace{0.1 cm}
\{ \phi \in H^{1} \left( \Omega^{\even}(M, E) \right) \mid  \left( I - {\widetilde P}(\theta) \right) \gamma^{\ast} (\phi|_{Y}) = 0 \}
\hspace{0.1 cm} = \hspace{0.1 cm} \Dom \left( \B_{{\widetilde P}(\theta)} \right).
$$

\noindent
Hence, it's enough to show that $\B_{{\widetilde P}(\theta)}$ is a symmetric operator.
For $\phi$, $\psi \in \Dom \left( \B_{{\widetilde P}(\theta)} \right)$,

\begin{eqnarray*}
\langle \B \phi, \hspace{0.1 cm} \psi \rangle_{M} - \langle \phi, \hspace{0.1 cm} \B \psi \rangle_{M} & = &
 \hspace{0.1 cm} \langle \phi|_{Y}, \hspace{0.1 cm} {\mathcal \gamma} \left( \psi|_{Y} \right) \rangle_{Y}  \nonumber \\
& = &   \hspace{0.1 cm} \langle (I - {\widetilde P}(\theta)) \phi|_{Y},
\hspace{0.1 cm} {\mathcal \gamma} (I - {\widetilde P}(\theta)) \left( \psi|_{Y} \right) \rangle_{Y}
\hspace{0.1 cm} = \hspace{0.1 cm}  \hspace{0.1 cm} \langle (I - {\widetilde P}(\theta)) \phi|_{Y},
\hspace{0.1 cm} {\widetilde P}(\theta) {\mathcal \gamma} \left( \psi|_{Y} \right) \rangle_{Y}  \nonumber \\
& = &   \hspace{0.1 cm} \langle {\widetilde P}(\theta) (I - {\widetilde P}(\theta)) \phi|_{Y},
\hspace{0.1 cm} {\mathcal \gamma} \left( \psi|_{Y} \right) \rangle_{Y}  \hspace{0.1 cm} = \hspace{0.1 cm}  0,
\end{eqnarray*}

\noindent
which completes the proof of the lemma.
\end{proof}

\vspace{0.2 cm}

Setting
$$
U(\theta) \hspace{0.1 cm} = \hspace{0.1 cm}
\left( \begin{array}{clcr} P(\theta)^{\ast} \hspace{0.1 cm} U_{\Pi_{>}} & 0 \\ 0 & \Id \end{array} \right) {\mathcal P}_{\ast}
+ \left( \Id - {\mathcal P}_{\ast} \right) \hspace{0.1 cm} = \hspace{0.1 cm}
\left( \begin{array}{clcr} cos \theta + U_{{\mathcal P}_{-}}^{\ast} U_{\Pi_{>}} sin \theta & 0 \\ 0 & \Id \end{array} \right)  {\mathcal P}_{\ast}
+ \left( \Id - {\mathcal P}_{\ast} \right),
$$

\noindent
it is straightforward that

\begin{equation}
U(\theta) {\widetilde P}(0) U(\theta)^{\ast} = {\widetilde P}(\theta).
\end{equation}

\noindent
Moreover, setting

\begin{equation}  \label{E:3.111}
\hspace{0.1 cm} T(\theta) = -i \hspace{0.1 cm}  \theta \left( \begin{array}{clcr} U_{{\mathcal P}_{-}}^{\ast} U_{\Pi_{>}} & 0 \\ 0 & 0 \end{array} \right) {\mathcal P}_{\ast},
\end{equation}

\noindent
$T(\theta)$ is a self-adjoint operator and we have

\begin{equation}
exp\{ i T(\theta) \} = U(\theta).
\end{equation}

\noindent
$T(\theta)$ satisfies the following property.

\vspace{0.2 cm}

\begin{lemma} \label{Lemma:3.22}
$T(\theta)$ commutes with ${\mathcal \gamma}$ and $\B_{Y}^{2}$, {\it i.e.},
\begin{equation}   \label{E:3.7}
{\mathcal \gamma} \hspace{0.1 cm} T(\theta) \hspace{0.1 cm} = \hspace{0.1 cm} T(\theta) \hspace{0.1 cm} {\mathcal \gamma}, \qquad
\B_{Y}^{2} \hspace{0.1 cm} T(\theta) \hspace{0.1 cm} = \hspace{0.1 cm} T(\theta) \hspace{0.1 cm} \B_{Y}^{2}.
\end{equation}
\end{lemma}

\vspace{0.2 cm}

\noindent
{\it Remark} : Contrary to the case of [6], $T(\theta)$ does not anticommute with ${\mathcal A}$.

\vspace{0.2 cm}

Let $\phi : [0, 1] \rightarrow [0, 1]$ be a decreasing smooth function such that $\phi = 1$ on a small neighborhood of $0$ and
$\phi = 0$ on a small neighborhood of $1$. We use this cut-off function to extend $T(\theta)$ defined on $\Omega^{\even}(M, E)|_{Y}$
to an operator defined on $\Omega^{\even}(M, E)$.
We define
$\Psi_{\theta} : \Omega^{\even}(M, E) \rightarrow \Omega^{\even}(M, E)$ by

\begin{equation} \label{E:3.8}
\Psi_{\theta} (\omega) (x) = e^{i \phi(x) T(\theta)} \omega (x),
\end{equation}

\noindent
where the support of $\phi(x) T(\theta)$ is contained in $N$, the collar neighborhood of $Y$.

\vspace{0.2 cm}

\begin{lemma} \label{Lemma:3.3}
$\Psi_{\theta}$ is a unitary operator mapping from $\Dom \left(\B_{{\widetilde P}(0)} \right)$ onto $\Dom \left(\B_{{\widetilde P}(\theta)} \right)$.
\end{lemma}

\begin{proof} : Clearly $\Psi_{\theta}$ is a unitary operator.
Let ${\widetilde P}(0) \omega(0) = 0$. Then
\begin{eqnarray*}
{\widetilde P}(\theta) (\Psi_{\theta} \omega) (0) & = & U(\theta) {\widetilde P}(0) U(\theta)^{\ast} \left( e^{i \phi(x) T(\theta)} \omega \right)|_{x=0}  \\
& = & U(\theta) {\widetilde P}(0) e^{- i T(\theta)} \left( e^{i \phi(x) T(\theta)} \omega \right)|_{x=0} \hspace{0.1 cm} = \hspace{0.1 cm}
U(\theta) \left( {\widetilde P}(0) \omega(0) \right) \hspace{0.1 cm} = \hspace{0.1 cm} 0,
\end{eqnarray*}
which completes the proof of the lemma.
\end{proof}

\vspace{0.2 cm}

\noindent
We now consider the following diagram.

\vspace{0.3 cm}

$$\begin{CD}
  \Dom \left( \B_{{\widetilde P}(0)} \right)     & @> \B_{{\widetilde P}(0)} >> &  \Omega^{\even}(M, E)       \\
  @V \Psi_{\theta} VV      &        & @V{\Psi_{\theta}}VV    \\
  \Dom \left( \B_{{\widetilde P}(\theta)} \right)     & @> \B_{{\widetilde P}(\theta)}  >> & \Omega^{\even}(M, E)
\end{CD}
$$

\vspace{0.3 cm}

\noindent
Setting $\B(\theta) := \Psi_{\theta}^{\ast} \B_{{\widetilde P}(\theta)} \Psi_{\theta}$,
$$
\B(\theta) : \Dom \left(\B_{{\widetilde P}(0)} \right) \rightarrow \Omega^{\even}(M, E)
$$
is an elliptic $\Psi$DO of order $1$
with a fixed domain $\Dom \left(\B_{{\widetilde P}(0)} \right)$ and have the same spectrum as $\B_{{\widetilde P}(\theta)}$.

\vspace{0.2 cm}

We next discuss one parameter family of eta functions $\eta_{\B(\theta)}(s)$ defined by

\vspace{0.2 cm}

\begin{equation} \label{E:3.9}
\eta_{\B(\theta)}(s) = \frac{1}{\Gamma(\frac{s+1}{2})} \int_{0}^{\infty} t^{\frac{s-1}{2}} \Tr \left( \B(\theta) e^{-t \B(\theta)^{2}} \right) dt.
\end{equation}

\vspace{0.2 cm}

\noindent
If $\eta_{\B(\theta)}(s)$ has a regular value at $s=0$, we define the eta invariant $\eta (\B(\theta))$ by

\begin{equation}  \label{E:3.99}
\eta (\B(\theta)) = \frac{1}{2} \left( \eta_{\B(\theta)}(0) + \Dim \Ker \B(\theta) \right).
\end{equation}

\vspace{0.2 cm}

\noindent
For $0 \leq \theta_{0} \leq \frac{\pi}{2}$, there exist $c(\theta_{0}) > 0$ and $\delta > 0$ such that
$c(\theta_{0}) \notin \Spec \left( \B_{\theta} \right)$ for $\theta_{0} - \delta < \theta < \theta_{0} + \delta$.
We denote by $Q(\theta)$ the orthogonal projection onto the space spanned by eigensections of $\B(\theta)$ whose eigenvalues are
less than $c(\theta)$ for $\theta_{0} - \delta < \theta < \theta_{0} + \delta$.
We define

\begin{eqnarray*}
\eta_{\B(\theta)} \left( s \hspace{0.1 cm} ; \hspace{0.1 cm} c(\theta) \right) \hspace{0.1 cm} =
\sum_{| \lambda | > c(\theta)} sign (\lambda) \vert \lambda \vert^{-s}
\hspace{0.1 cm} = \hspace{0.1 cm} \frac{1}{\Gamma(\frac{s+1}{2})} \int^{\infty}_{0} t^{\frac{s-1}{2}}
\Tr \left\{ \left( I - Q(\theta) \right) \B(\theta) e^{- t \B(\theta)^{2}} \right\} dt.
\end{eqnarray*}

\noindent
Then $\eta_{\B(\theta)} (s) - \eta_{\B(\theta)} \left( s \hspace{0.1 cm} ; c(\theta) \right)$ is an entire function and

\vspace{0.2 cm}

\begin{equation}   \label{E:3.155}
\left\{ \frac{1}{2} \left( \eta_{\B(\theta)} (s) + \Dim \Ker \B(\theta) \right) - \frac{1}{2} \eta_{\B(\theta)} \left( s \hspace{0.1 cm} ; c(\theta) \right) \right\}_{s=0}
\end{equation}

\vspace{0.2 cm}

\noindent
does not depend on $\theta$ for $\theta_{0} - \delta < \theta < \theta_{0} + \delta$ up to $\Mod {\Bbb Z}$.
Simple computation shows that

\begin{eqnarray}  \label{E:3.144}
& & \frac{\partial}{\partial \theta} \eta_{\B(\theta)} \left( s \hspace{0.1 cm} ; c(\theta) \right)  \\
& = & \frac{1}{\Gamma(\frac{s+1}{2})} \int^{\infty}_{0} t^{\frac{s-1}{2}} \Tr \left( - {\dot Q}(\theta) \B(\theta) e^{-t \B(\theta)^{2}}  +
 \left( I - Q(\theta) \right) \frac{\partial}{\partial \theta} \left( \B(\theta) e^{- t \B(\theta)^{2}} \right) \right) dt  \nonumber \\
& = & \frac{1}{\Gamma(\frac{s+1}{2})} \int^{\infty}_{0} t^{\frac{s-1}{2}} \Tr \left( - {\dot Q}(\theta) \B(\theta) e^{-t \B(\theta)^{2}} \right) dt -
\frac{s}{\Gamma(\frac{s+1}{2})} \int^{\infty}_{0} t^{\frac{s-1}{2}}
\Tr \left\{ \left( I - Q(\theta) \right) \left( {\dot \B}(\theta) e^{- t \B(\theta)^{2}} \right) \right\} dt \nonumber,
\end{eqnarray}

\vspace{0.2 cm}

\noindent
where ${\dot Q(\theta)}$ and ${\dot \B(\theta)}$ are derivatives of $Q(\theta)$ and $\B(\theta)$ with respect to $\theta$.
Furthermore, we have (cf. [14])
$$
\Tr \left( - {\dot Q}(\theta) \B(\theta) e^{-t \B(\theta)^{2}} \right) = 0, \qquad
\left\{ \frac{s}{\Gamma(\frac{s+1}{2})} \int^{\infty}_{0} t^{\frac{s-1}{2}}
\Tr \left( Q(\theta)  {\dot \B}(\theta) e^{- t \B(\theta)^{2}} \right) dt \right\}_{s=0} = 0.
$$

\noindent
These equalities imply that

\vspace{0.2 cm}

\begin{eqnarray}  \label{E:3.156}
\frac{\partial}{\partial \theta} \eta_{\B(\theta)} \left( s \hspace{0.1 cm} ; c(\theta) \right) & = &
- \hspace{0.1 cm} \frac{s}{\Gamma(\frac{s+1}{2})} \int^{\infty}_{0} t^{\frac{s-1}{2}}
\Tr \left( {\dot \B}(\theta) e^{- t \B(\theta)^{2}} \right) dt + F(s) ,
\end{eqnarray}

\vspace{0.2 cm}
\noindent
where $F(s)$ is an analytic function at least for $\re s > -1$ with $F(0) = 0$.

\vspace{0.3 cm}

Recall that
\begin{equation}  \label{E:3.11}
\B(\theta) \hspace{0.1 cm} = \hspace{0.1 cm} \Psi_{\theta}^{\ast} \B_{{\widetilde P}(\theta)} \Psi_{\theta}
\hspace{0.1 cm} = \hspace{0.1 cm} e^{- i \phi(x) T(\theta)} {\mathcal \gamma} (\partial_{x} + {\mathcal A}) e^{i \phi(x) T(\theta)}.
\end{equation}
Using the fact that $T(\theta) T^{\prime}(\theta) = T^{\prime}(\theta) T(\theta)$ and Lemma \ref{Lemma:3.22}, we have

\begin{equation}  \label{E:3.12}
{\dot \B}(\theta) \hspace{0.1 cm} = \hspace{0.1 cm}
e^{- i \phi(x) T(\theta)}  \left( i \phi^{\prime}(x) {\mathcal \gamma} T^{\prime}(\theta) -
i \phi(x) {\mathcal \gamma} T^{\prime}(\theta) {\mathcal A} + i \phi(x) {\mathcal \gamma} {\mathcal A} T^{\prime}(\theta)
\right) e^{i \phi(x) T(\theta)},
\end{equation}

\noindent
which leads to

\begin{equation} \label{E:3.13}
\Tr \left( {\dot \B(\theta)} e^{-t \B(\theta)^{2}} \right) \hspace{0.1 cm} = \hspace{0.1 cm}
\Tr \left\{ \left( i \phi^{\prime}(x) {\mathcal \gamma} T^{\prime}(\theta) -
i \phi(x) {\mathcal \gamma} [ T^{\prime}(\theta), \hspace{0.1 cm} {\mathcal A} ] \hspace{0.1 cm} \right) e^{-t \B^{2}_{{\widetilde P}(\theta)}} \right\}.
\end{equation}

\vspace{0.2 cm}

\noindent
Since the support of $\phi$ is in $[0, 1]$, the support of ${\dot \B(\theta)}$ is in $[0, 1] \times Y$.
Let $\B^{\cyl}$ be the odd signature operator defined by (\ref{E:1.8}) on $[0, \infty) \times Y$.
The heat kernel of $\left( \B^{\cyl}_{{\widetilde P}(\theta)} \right)^{2}$ was computed in [6] as follows.

\vspace{0.2 cm}

\begin{eqnarray} \label{E:3.14}
e^{-t \left( \B^{\cyl}_{{\widetilde P}(\theta)} \right)^{2}} (x, y)  & = & (4 \pi t)^{- \frac{1}{2}} \left( e^{- \frac{(x-y)^{2}}{4t}} +
(I - 2 {\widetilde P} (\theta)) e^{- \frac{(x+y)^{2}}{4t}} \right) e^{-t {\mathcal A}^{2}}  \nonumber \\
& & + (\pi t)^{- \frac{1}{2}} \left( I - {\widetilde P} (\theta) \right) \int_{0}^{\infty} e^{- \frac{(x+y+z)^{2}}{4t}} {\widetilde {\mathcal A}}(\theta)
e^{{\widetilde {\mathcal A}}(\theta)z - t {\mathcal A}^{2}} dz,
\end{eqnarray}

\vspace{0.2 cm}

\noindent
where
${\widetilde {\mathcal A}}(\theta) := (I - {\widetilde P}(\theta) ) {\mathcal A} (I - {\widetilde P } (\theta) )$.
The standard theory for heat kernel ([1], [3]) implies
that the asymptotic expansions of $\Tr \left( {\dot \B(\theta)} e^{-t \B(\theta)^{2}} \right)$ is equal
to that of $\Tr \left( {\dot \B^{\cyl}(\theta)} e^{-t \left( \B^{\cyl} (\theta) \right)^{2}} \right)$ up to
$\left( e^{- \frac{c}{t}} \right)$ for some $c > 0$.
With a little abuse of notation we write $\B^{\cyl}$ by $\B$ again.
Equation (\ref{E:3.13}) leads to the following equality.

\begin{eqnarray} \label{E:3.15}
& & \Tr \left( i \phi^{\prime}(x) {\mathcal \gamma} T^{\prime}(\theta) e^{-t \B^{2}_{{\widetilde P}(\theta)}} \right)     \\
& = & \frac{i}{\sqrt{4 \pi t}}  \int_{0}^{\infty} \phi^{\prime}(x) dx \hspace{0.1 cm}
\Tr \left( {\mathcal \gamma} T^{\prime}(\theta) e^{-t {\mathcal A}^{2}} \right)
\nonumber \\
& + & \frac{i}{\sqrt{4 \pi t}}  \int_{0}^{\infty} \phi^{\prime}(x) e^{- \frac{x^{2}}{t}} dx \hspace{0.1 cm}
\Tr \left( {\mathcal \gamma} T^{\prime}(\theta) (I - 2 {\widetilde P}(\theta)) e^{-t {\mathcal A}^{2}} \right) \nonumber \\
& + & \frac{i}{\sqrt{\pi t}} \int_{0}^{\infty} \int_{0}^{\infty} \phi^{\prime}(x) e^{- \frac{(2x + z)^{2}}{4t}}
\Tr \left\{ {\mathcal \gamma} T^{\prime}(\theta) (I - {\widetilde P}(\theta))  {\widetilde {\mathcal A}}(\theta)
e^{{\widetilde {\mathcal A}}(\theta)z - t {\mathcal A}^{2}} \right\} dz dx . \nonumber
\end{eqnarray}

\vspace{0.2 cm}

\noindent
Lemma \ref{Lemma:3.2} and Lemma \ref{Lemma:3.22} imply that
$\Tr \left( {\mathcal \gamma} T^{\prime}(\theta) (I - 2 {\widetilde P}(\theta)) e^{-t {\mathcal A}^{2}} \right) = 0$.
Since $\phi (x) = 1$ near $x=0$, the third integral decays exponentially as $t \rightarrow 0^{+}$.
Hence,

\begin{equation} \label{E:3.16}
\Tr \left( i \phi^{\prime}(x) {\mathcal \gamma} T^{\prime}(\theta) e^{-t \B^{2}_{{\widetilde P}(\theta)}} \right) \hspace{0.1 cm} = \hspace{0.1 cm}
\frac{-i}{\sqrt{4 \pi t}} \Tr \left( {\mathcal \gamma} T^{\prime}(\theta) e^{-t {\mathcal A}^{2}} \right) + O(e^{- \frac{c}{t}}).
\end{equation}

\noindent
We refer to p.456 in [6] for the proof of the following equality.

\begin{equation}  \label{E:3.31}
{\widetilde {\mathcal A}}(\theta) e^{{\widetilde {\mathcal A}}(\theta)z - t {\mathcal A}^{2}} \hspace{0.1 cm} = \hspace{0.1 cm}
(- cos\theta) | {\mathcal A}| (I - {\widetilde P}(\theta)) e^{- cos\theta |{\mathcal A}| z - t {\mathcal A}^{2} }.
\end{equation}

\vspace{0.2 cm}

\noindent
Then, we have

\begin{eqnarray} \label{E:3.17}
& & \Tr \left( - i \phi (x) {\mathcal \gamma} [T^{\prime}(\theta), \hspace{0.1 cm} {\mathcal A}] e^{-t \B^{2}_{{\widetilde P}(\theta)}} \right)   \\
& = & \frac{- i}{\sqrt{4 \pi t}}  \int_{0}^{\infty} \phi(x) dx \hspace{0.1 cm} \Tr \left( {\mathcal \gamma}
[T^{\prime}(\theta), \hspace{0.1 cm} {\mathcal A}] e^{-t {\mathcal A}^{2}} \right)      \nonumber  \\
& + & \frac{- i}{\sqrt{4 \pi t}}  \int_{0}^{\infty} \phi(x) e^{- \frac{x^{2}}{t}} dx
\hspace{0.1 cm} \Tr \left( {\mathcal \gamma} [T^{\prime}(\theta), \hspace{0.1 cm} {\mathcal A}]
(I - 2 {\widetilde P}(\theta)) e^{-t {\mathcal A}^{2}} \right) \nonumber \\
& + & \frac{- i}{\sqrt{\pi t}} \int_{0}^{\infty} \int_{0}^{\infty} \phi(x) e^{- \frac{(2x + z)^{2}}{4t}}
\hspace{0.1 cm} \Tr \left\{ {\mathcal \gamma} [T^{\prime}(\theta), \hspace{0.1 cm} {\mathcal A}]
(I - {\widetilde P}(\theta)) (- cos\theta) |{\mathcal A}| (I - {\widetilde P}(\theta))
e^{- cos\theta |{\mathcal A}| z - t {\mathcal A}^{2} } \right\} dz dx  \nonumber  \\
& = : & (I) + (II) + (III).   \nonumber
\end{eqnarray}

\vspace{0.2 cm}

\noindent
Lemma \ref{Lemma:3.22} shows that
$\Tr \left( {\mathcal \gamma} [T^{\prime}(\theta), \hspace{0.1 cm} {\mathcal A}] e^{-t {\mathcal A}^{2}} \right) = 0$, which yields

\begin{equation} \label{E:3.18}
(I) \hspace{0.1 cm} = \hspace{0.1 cm} 0, \qquad
(II) \hspace{0.1 cm} = \hspace{0.1 cm}
\frac{i}{2} \Tr \left(  {\mathcal \gamma} [T^{\prime}(\theta), \hspace{0.1 cm} {\mathcal A}] {\widetilde P}(\theta) e^{-t {\mathcal A}^{2}}  \right)
+  O(e^{- \frac{c}{t}}).
\end{equation}

\vspace{0.2 cm}

\noindent
Change of variables, Lemma \ref{Lemma:3.2} and Lemma \ref{Lemma:3.22} show that

\begin{equation} \label{E:3.19}
(III)  \hspace{0.1 cm} = \hspace{0.1 cm}
\frac{ - 2 i cos\theta}{\sqrt{\pi}} \sqrt{t} \int_{0}^{\infty} \int_{0}^{\infty} \phi(\sqrt{t}x) e^{- (x + z)^{2}} \Tr \left( {\mathcal \gamma}
[T^{\prime}(\theta), \hspace{0.1 cm} {\mathcal A}] {\widetilde P}(\theta) |{\mathcal A}|
e^{-2 \sqrt{t} cos\theta |{\mathcal A}| z - t {\mathcal A}^{2}} \right) dx dz .
\end{equation}

\noindent
Since $|{\mathcal A}|$ commutes with ${\mathcal A}$, $T^{\prime}(\theta)$ and ${\widetilde P}(\theta)$, we denote
\begin{equation} \label{E:3.20}
d(\lambda) := \Tr_{\Ker ( |{\mathcal A}| - \lambda )} \left( {\mathcal \gamma} [T^{\prime}(\theta), \hspace{0.1 cm} {\mathcal A}] {\widetilde P} (\theta) \right).
\end{equation}

\begin{eqnarray} \label{E:3.21}
(III) & = & -  i cos\theta \sum_{0 \neq \lambda \in \Spec(|{\mathcal A}|)} d(\lambda) \int_{0}^{\infty} \int_{0}^{\infty} \phi(\sqrt{t}x)
\frac{2}{\sqrt{\pi}} \sqrt{t} \lambda e^{-(x+z)^{2}} e^{-2 cos\theta \sqrt{t} \lambda z - t \lambda^{2}} dx dz  \nonumber  \\
& = & -  i cos\theta \sum_{0 \neq \lambda \in \Spec(|{\mathcal A}|)} d(\lambda) \int_{0}^{\infty}
\frac{2}{\sqrt{\pi}} \int_{0}^{\infty} e^{-(x+z)^{2}} dx \sqrt{t} \lambda e^{-2 cos\theta \sqrt{t} \lambda z - t \lambda^{2}}dz
+ O(e^{- \frac{c}{t}}) \nonumber  \\
& = & -  i cos\theta \sum_{0 \neq \lambda \in \Spec(|{\mathcal A}|)} d(\lambda) \int_{0}^{\infty} erfc(z)
 \sqrt{t} \lambda e^{-2 cos\theta \sqrt{t} \lambda z - t \lambda^{2}}dz + O(e^{- \frac{c}{t}}) ,
\end{eqnarray}

\noindent
where $erfc(x) := \frac{2}{\sqrt{\pi}} \int_{x}^{\infty} e^{-y^{2}} dy$. To compute $(III)$ more precisely we introduce the following concepts.

\vspace{0.3 cm}

Let $A$, $B$ be classical pseudodifferential operators and $A$ be an elliptic operator of positive order on a compact manifold.
Then $\Tr \left( Be^{-t A^{2}} \right)$ has an asymptotic expansion of the following type for $t \rightarrow 0^{+}$.

\begin{equation} \label{E:3.26}
\Tr \left( Be^{-t A^{2}} \right) \sim \sum_{\re \alpha \rightarrow \infty} a_{\alpha, k}(A, B) t^{\alpha} (\log t)^{k}.
\end{equation}

\noindent
When $B$ commutes with $A^{2}$ and vanishes on $\Ker A^{2}$,
we define the eta function $\eta(A, B \hspace{0.1 cm} ; \hspace{0.1 cm} s)$ by

\begin{eqnarray}  \label{E:3.27}
\eta(A, B \hspace{0.1 cm} ; \hspace{0.1 cm} s) & := &
\frac{1}{\Gamma(\frac{s+1}{2})} \int_{0}^{\infty} t^{\frac{s-1}{2}} \Tr \left( Be^{-t A^{2}} \right) dt  \nonumber \\
& = & \sum_{|\lambda| \in \Spec(|A|) - \{ 0 \}} \left( \Tr_{\Ker(|A| - |\lambda|)} B \right) | \lambda |^{-s-1}.
\end{eqnarray}

\noindent
Then the noncommutative residue $\rres$ is defined as follows ([30], [31], [15]).

\begin{equation} \label{E:3.28}
\rres (B) := -2 \ord (A) \hspace{0.1 cm} a_{0,1}(A, B) \hspace{0.1 cm} = \hspace{0.1 cm} \ord(A) \hspace{0.1 cm} \Res_{s=-1}
\eta(A, B \hspace{0.1 cm} ; \hspace{0.1 cm} s).
\end{equation}

\noindent
The following is well known ([30], [6]).

\begin{lemma} \label{Lemma:3.7}
If $B$ is a classical pseudodifferential operator on a compact manifold with $B^{2} = B$, then $\rres (B) = 0$.
\end{lemma}

We now go back to (\ref{E:3.21}).
We define a function $F_{\theta}(x)$ and its Mellin transform ${\mathcal M}F_{\theta}(s)$
( see [6] for details ) by

\begin{equation} \label{E:3.22}
F_{\theta}(x) = x \int_{0}^{\infty} erfc(z) e^{-2 cos\theta x z - x^{2}} dz, \qquad
{\mathcal M}F_{\theta}(s) = \int_{0}^{\infty} x^{s-1} F_{\theta}(x) dx.
\end{equation}

\noindent
Using the inverse Mellin transform, we have

\begin{eqnarray}  \label{E:3.23}
(III) & = & -  i cos\theta \sum_{0 \neq \lambda \in \Spec(|{\mathcal A}|)} d(\lambda) \hspace{0.1 cm} F_{\theta}(\sqrt{t} \lambda) \hspace{0.1 cm}
+ \hspace{0.1 cm}  O(e^{- \frac{c}{t}}) \nonumber  \\
& = & -  i cos\theta \sum_{0 \neq \lambda \in \Spec(|{\mathcal A}|)} d(\lambda) \hspace{0.1 cm} \frac{1}{2 \pi i} \int_{\re w = c \gg 0} (\sqrt{t} \lambda)^{- w}
\hspace{0.1 cm} {\mathcal M} F_{\theta}(w) \hspace{0.1 cm} dw \hspace{0.1 cm} + \hspace{0.1 cm} O(e^{- \frac{c}{t}}) \nonumber  \\
& = & \frac{-  i cos\theta}{2 \pi i} \int_{\re w = c \gg 0} t^{- \frac{w}{2}} \sum_{0 \neq \lambda \in \Spec(|{\mathcal A}|)} d(\lambda)
\hspace{0.1 cm}
\lambda^{- w}  \hspace{0.1 cm} {\mathcal M} F_{\theta}(w) \hspace{0.1 cm} dw  \hspace{0.1 cm}  +  \hspace{0.1 cm} O(e^{- \frac{c}{t}}) \nonumber  \\
& = & \frac{-  i cos\theta}{2 \pi i} \int_{\re w = c \gg 0} t^{- \frac{w}{2}}  \hspace{0.1 cm}
\eta \left( {\mathcal A}, {\mathcal \gamma} [ T^{\prime}(\theta), {\mathcal A} ] {\widetilde P} (\theta) \hspace{0.1 cm} ; \hspace{0.1 cm} w-1 \right)
\hspace{0.1 cm} {\mathcal M} F_{\theta}(w) \hspace{0.1 cm} dw \hspace{0.1 cm}  +  \hspace{0.1 cm} O(e^{- \frac{c}{t}}) \nonumber  \\
& = & -  i cos\theta \cdot \Res_{\re w < c} \left(  t^{- \frac{w}{2}} \hspace{0.1 cm}
\eta \left( {\mathcal A}, {\mathcal \gamma} [ T^{\prime}(\theta), {\mathcal A} ] {\widetilde P} (\theta)
\hspace{0.1 cm} ; \hspace{0.1 cm} w-1 \right) \hspace{0.1 cm} {\mathcal M} F_{\theta}(w) \right)
\hspace{0.1 cm} + \hspace{0.1 cm} O(e^{- \frac{c}{t}}) .
\end{eqnarray}

\noindent
Equations (\ref{E:3.17}), (\ref{E:3.18}) and (\ref{E:3.23}) show that

\begin{eqnarray} \label{E:3.24}
& & \Tr \left( - i \phi (x) {\mathcal \gamma} [T^{\prime}(\theta), \hspace{0.1 cm} {\mathcal A}] e^{-t \B^{2}_{{\widetilde P}(\theta)}} \right)  =
\frac{i}{2} \Tr \left(  {\mathcal \gamma} [T^{\prime}(\theta), \hspace{0.1 cm} {\mathcal A}]
{\widetilde P}(\theta) e^{-t {\mathcal A}^{2}}  \right)  \nonumber \\
& - &  i cos\theta \cdot \Res_{\re w < c} \left(  t^{- \frac{w}{2}}
\eta \left( {\mathcal A}, {\mathcal \gamma} [T^{\prime}(\theta), \hspace{0.1 cm}
{\mathcal A}] {\widetilde P} (\theta) \hspace{0.1 cm} ; \hspace{0.1 cm} w-1 \right) {\mathcal M} F_{\theta}(w) \right) +
O(e^{- \frac{c}{t}}) .
\end{eqnarray}

\noindent
Equations (\ref{E:3.13}), (\ref{E:3.16}) and (\ref{E:3.24}) lead to the following lemma.

\vspace{0.2 cm}

\begin{lemma}  \label{Lemma:3.4}

\begin{eqnarray*}
\Tr \left( {\dot \B(\theta)} e^{-t \B(\theta)^{2}} \right)  & = &
\frac{-i}{\sqrt{4 \pi t}} \Tr \left( {\mathcal \gamma} T^{\prime}(\theta) e^{-t {\mathcal A}^{2}} \right)
+ \frac{i}{2} \Tr \left(  {\mathcal \gamma} [T^{\prime}(\theta), \hspace{0.1 cm} {\mathcal A}] {\widetilde P}(\theta) e^{-t {\mathcal A}^{2}}  \right)  \nonumber \\
& - & i cos\theta \cdot \Res_{\re w < c} \left(  t^{- \frac{w}{2}}
\eta \left({\mathcal A}, {\mathcal \gamma} [T^{\prime}(\theta), \hspace{0.1 cm} {\mathcal A}] {\widetilde P} (\theta) ; w-1 \right)
{\mathcal M} F_{\theta}(w) \right) + O(e^{- \frac{c}{t}}) .
\end{eqnarray*}

\end{lemma}

\vspace{0.2 cm}

It is known  that (\ref{E:3.9}) has at most a simple pole at $s=0$ (Theorem 3.4 in [6]) and has regular values at
$s=0$ for $\theta = 0$ and $\frac{\pi}{2}$ (for the case of $\theta = \frac{\pi}{2}$, see [14]).
Moreover, ${\mathcal M} F_{\theta}(w)$ has only simple poles at negative integers (Lemma 3.3 in [6]).
The following lemma is due to [13] (cf. [6]).

\vspace{0.2 cm}

\begin{lemma} \label{Lemma:3.5}
Let $A$ and $B$ be classical pseudodifferential operators of order $a$ and $b$, respectively, on a compact manifold $M$ with $\Dim M = m$.
If $A$ is a self-adjoint elliptic operator of positive order, then for $t \rightarrow 0^{+}$,
$$
\Tr \left( B e^{-t A^{2}} \right) \sim \sum_{j=0}^{\infty} a_{j}(A, B) t^{\frac{j-m-b}{2a}} + \sum_{j=0}^{\infty} \left( b_{j}(A, B) \log t +
c_{j}(A, B) \right) t^{j}.
$$
\end{lemma}

\vspace{0.2 cm}

\noindent
The equation (\ref{E:3.156}) with
Lemma \ref{Lemma:3.4} and Lemma \ref{Lemma:3.5}  (cf. Theorem 3.4 and 3.5 in [6]) implies that

\begin{equation} \label{E:3.25}
\frac{d}{d \theta} \Res_{s=0} \eta_{\B(\theta)}(s) \hspace{0.1 cm} = \hspace{0.1 cm} \Res_{s=0} \left( \frac{d}{d \theta} \eta_{\B(\theta)}(s) \right)
\hspace{0.1 cm} = \hspace{0.1 cm}
\frac{4}{\sqrt{\pi}} \hspace{0.1 cm} a_{- \frac{1}{2}, 1} (\B(\theta), {\dot B}(\theta))
\hspace{0.1 cm} = \hspace{0.1 cm}
\frac{1}{\pi} \rres \left( i {\mathcal \gamma} T^{\prime}(\theta) \right),
\end{equation}

\vspace{0.1 cm}

\noindent
where $a_{- \frac{1}{2}, 1} (\B(\theta), {\dot B}(\theta))$ is the coefficient of $t^{- \frac{1}{2}} \log t$ in the asymptotic expansion of
$\Tr \left( {\dot \B(\theta)} e^{-t \B(\theta)^{2}} \right) $ for $t \rightarrow 0^{+}$.

\vspace{0.2 cm}

\begin{lemma} \label{Lemma:3.6}
\begin{eqnarray*}
\Tr \left( i \hspace{0.1 cm} {\mathcal \gamma} T^{\prime}(\theta) \hspace{0.1 cm} e^{-t {\mathcal A}^{2}} \right) \hspace{0.1 cm} = \hspace{0.1 cm} 0.
\end{eqnarray*}

\noindent
Hence,
$\rres \left( i {\mathcal \gamma} T^{\prime}(\theta) \right) = 0 $
and $\eta_{\B(\theta)}(s)$ has a regular value at $s=0$ for each $0 \leq \theta \leq \frac{\pi}{2}$.
\end{lemma}

\begin{proof}
We recall that $T^{\prime}(\theta) = - i \left( \begin{array}{clcr} U_{{\mathcal P}_{-}}^{\ast} U_{\Pi_{>}} & 0 \\ 0 & 0 \end{array} \right)
{\mathcal P}_{\ast}$.
Using (3) in Lemma \ref{Lemma:3.1}, we have

$$
 \Tr \left( i \hspace{0.1 cm} {\mathcal \gamma} T^{\prime}(\theta) \hspace{0.1 cm} e^{-t {\mathcal A}^{2}} \right)  =
\Tr \left( {\mathcal \gamma} \left( \begin{array}{clcr} U_{{\mathcal P}_{-}}^{\ast} U_{\Pi_{>}} & 0 \\ 0 & 0 \end{array} \right)
{\mathcal P}_{\ast} e^{- t {\mathcal A}^{2}}
\right)
\hspace{0.1 cm} = \hspace{0.1 cm}
\Tr \left( {\mathcal \gamma} \left( \begin{array}{clcr} 0 & U_{{\mathcal P}_{-}}^{\ast} \\ 0 & 0 \end{array} \right)
\left( \begin{array}{clcr} 0 & 0 \\ U_{\Pi_{>}} & 0 \end{array} \right) {\mathcal P}_{\ast} e^{- t {\mathcal A}^{2}} \right)
$$
$$
 =
\Tr \left( - {\mathcal \gamma} \left( \begin{array}{clcr} 0 & 0 \\ U_{\Pi_{>}} & 0 \end{array} \right)
\left( \begin{array}{clcr} 0 & U_{{\mathcal P}_{-}}^{\ast} \\ 0 & 0 \end{array} \right)
{\mathcal P}_{\ast} e^{- t {\mathcal A}^{2}} \right)
\hspace{0.1 cm} = \hspace{0.1 cm}
 \Tr \left( {\mathcal \gamma} \left( \begin{array}{clcr} 0 & 0 \\ 0 & U_{{\mathcal P}_{-}} U_{\Pi_{>}}^{\ast} \end{array} \right)
 {\mathcal P}_{\ast} e^{- t {\mathcal A}^{2}}
\right) .
$$

\vspace{0.2 cm}

\noindent
Since $\Gamma^{Y}$ anticommutes with  $\left( \left( \B_{Y}^{2} \right)^{-} - \left( \B_{Y}^{2} \right)^{+} \right)$,
we have, by (\ref{E:1.9}) and (\ref{E:3.54}),

\begin{eqnarray*}
& & \Tr \left( i \hspace{0.1 cm} {\mathcal \gamma} T^{\prime}(\theta) \hspace{0.1 cm} e^{-t {\mathcal A}^{2}} \right)   \hspace{0.1 cm} = \hspace{0.1 cm}
\frac{1}{2} \Tr \left( {\mathcal \gamma}
\left( \begin{array}{clcr} U_{{\mathcal P}_{-}}^{\ast} U_{\Pi_{>}} & 0 \\ 0 & U_{{\mathcal P}_{-}} U_{\Pi_{>}}^{\ast} \end{array} \right)
{\mathcal P}_{\ast} e^{- t {\mathcal A}^{2}} \right)   \\
& = &
\frac{1}{2} \Tr \left\{ i \beta \left( \left( \B_{Y}^{2} \right)^{-1} \left( \left( \B_{Y}^{2} \right)^{-} - \left( \B_{Y}^{2} \right)^{+} \right) \right)
\left( \B_{Y}^{2} \right)^{-\frac{1}{2}} \left( \nabla^{Y} \Gamma^{Y} + \Gamma^{Y} \nabla^{Y} \right) \left( \begin{array}{clcr} 0 & -1 \\ -1 & 0 \end{array}
\right) {\mathcal P}_{\ast} e^{- t {\mathcal A}^{2}} \right\}
\hspace{0.1 cm} = \hspace{0.1 cm} 0 ,
\end{eqnarray*}

\noindent
which completes the proof of the lemma.
\end{proof}

\vspace{0.2 cm}

\noindent
Since ${\mathcal M} F_{\theta}(w)$ has a regular value at $w=1$,
Lemma \ref{Lemma:3.4} and (\ref{E:3.156}) imply that

\begin{eqnarray}   \label{E:3.29}
& & \frac{d}{d \theta} \eta_{\B (\theta)}(0 \hspace{0.1 cm} ; c(\theta) ) \hspace{0.1 cm}  =  \hspace{0.1 cm}
- \frac{2}{\sqrt{\pi}} \hspace{0.1 cm} a_{- \frac{1}{2}, 0}(\B(\theta), \hspace{0.1 cm} {\dot \B}(\theta)) \\
& = &  - \frac{1}{\sqrt{\pi}} \hspace{0.1 cm} a_{- \frac{1}{2}, 0}
\left( {\mathcal A}, \hspace{0.1 cm}  i {\mathcal \gamma} [T^{\prime}(\theta), \hspace{0.1 cm} {\mathcal A}] {\widetilde P}(\theta) \right)
 + \hspace{0.1 cm} \frac{2}{\sqrt{\pi}} \hspace{0.1 cm} cos\theta \cdot {\mathcal M}F_{\theta}(1) \Res_{w=1}
\left( \eta \left( {\mathcal A}, \hspace{0.1 cm}
i {\mathcal \gamma} [T^{\prime}(\theta), \hspace{0.1 cm} {\mathcal A}]  {\widetilde P}(\theta) \hspace{0.1 cm} ; \hspace{0.1 cm} w-1 \right) \right). \nonumber
\end{eqnarray}

\noindent
We note that

\begin{eqnarray}  \label{E:3.30}
a_{- \frac{1}{2}, 0} \left( {\mathcal A}, \hspace{0.1 cm} i {\mathcal \gamma} [T^{\prime}(\theta), \hspace{0.1 cm} {\mathcal A}]  {\widetilde P}(\theta) \right)
& = & \frac{\sqrt{\pi}}{2} \Res_{s=0}
\left( \eta \left( {\mathcal A}, \hspace{0.1 cm}  i {\mathcal \gamma} [T^{\prime}(\theta), \hspace{0.1 cm} {\mathcal A}]
{\widetilde P}(\theta) \hspace{0.1 cm} ; \hspace{0.1 cm} s \right) \right) \nonumber  \\
& = & \frac{\sqrt{\pi}}{2} \Res_{s=0}
\left( \eta \left( {\mathcal A}, \hspace{0.1 cm}  i {\mathcal \gamma} [T^{\prime}(\theta), \hspace{0.1 cm} ( sign {\mathcal A} )]
{\widetilde P}(\theta) \hspace{0.1 cm} ; \hspace{0.1 cm} s-1 \right) \right) \nonumber  \\
& = & \frac{\sqrt{\pi}}{2} \rres \left( i {\mathcal \gamma} [T^{\prime}(\theta), \hspace{0.1 cm} ( sign {\mathcal A} )] {\widetilde P}(\theta) \right).
\end{eqnarray}

\noindent
Similarly,

\begin{eqnarray}   \label{E:3.31}
\Res_{w=1} \left( \eta({\mathcal A}, \hspace{0.1 cm}
i {\mathcal \gamma} [T^{\prime}(\theta), \hspace{0.1 cm} {\mathcal A}] {\widetilde P}(\theta) \hspace{0.1 cm} ; \hspace{0.1 cm} w-1 ) \right)
& = & \Res_{w=1}
\left( \eta \left( {\mathcal A}, \hspace{0.1 cm} i {\mathcal \gamma} [T^{\prime}(\theta), \hspace{0.1 cm} ( sign {\mathcal A} )]
{\widetilde P}(\theta) \hspace{0.1 cm} ; \hspace{0.1 cm} w-2 \right) \right)  \nonumber \\
& = & \Res_{w=-1}
\left( \eta \left( {\mathcal A}, \hspace{0.1 cm} i {\mathcal \gamma} [T^{\prime}(\theta), \hspace{0.1 cm} ( sign {\mathcal A} )]
{\widetilde P}(\theta) \hspace{0.1 cm} ; \hspace{0.1 cm} w \right) \right)   \nonumber \\
& = & \rres \left( i {\mathcal \gamma} [T^{\prime}(\theta), \hspace{0.1 cm} ( sign {\mathcal A} )] {\widetilde P}(\theta) \right).
\end{eqnarray}

\vspace{0.2 cm}

\noindent
The following lemma is straightforward.
\vspace{0.2 cm}

\begin{lemma}  \label{Lemma:3.10}
\begin{eqnarray*}
T^{\prime}(\theta) {\widetilde P}(\theta) \hspace{0.1 cm} = \hspace{0.1 cm}
\left( I - {\widetilde P}(\theta) \right) T^{\prime}(\theta) -
\frac{i}{2} \left( \begin{array}{clcr} 0 & {\mathcal W}^{\ast} \\ - {\mathcal W} & 0 \end{array} \right) {\mathcal P}_{\ast},
\end{eqnarray*}
where ${\mathcal W} := - U_{\Pi_{>}} sin\theta + U_{{\mathcal P}_{-}} cos\theta$.
\end{lemma}

\vspace{0.2 cm}

\noindent
Equations (\ref{E:3.29}), (\ref{E:3.30}), (\ref{E:3.31}) and Lemma \ref{Lemma:3.10} lead to the following result.

\vspace{0.2 cm}

\begin{lemma}  \label{Lemma:3.9}
\begin{eqnarray*}
\frac{d}{d \theta} \eta_{\B (\theta)} (0 \hspace{0.1 cm} ; c(\theta) ) & = &
 \left( - \frac{1}{2}  +  \frac{2 cos\theta}{\sqrt{\pi}} \cdot {\mathcal M}F_{\theta}(1) \right)
\rres \left( i {\mathcal \gamma} [T^{\prime}(\theta), \hspace{0.1 cm} ( sign {\mathcal A} )] {\widetilde P}(\theta) \right)
\hspace{0.1 cm} = \hspace{0.1 cm} 0.
\end{eqnarray*}
\end{lemma}

\begin{proof}
We note that
$$
\rres \left( i {\mathcal \gamma} [T^{\prime}(\theta), \hspace{0.1 cm} (sign {\mathcal A})] {\widetilde P}(\theta) \right) =
\rres \left( i {\mathcal \gamma} T^{\prime}(\theta) (sign {\mathcal A}) {\widetilde P}(\theta) \right) -
\rres \left( i {\mathcal \gamma} (sign {\mathcal A}) T^{\prime}(\theta) {\widetilde P}(\theta) \right).
$$
We are going to show that $\rres \left( i {\mathcal \gamma} T^{\prime}(\theta) (sign {\mathcal A}) {\widetilde P}(\theta) \right)  = 0$ and
$\rres \left( i {\mathcal \gamma} (sign {\mathcal A}) T^{\prime}(\theta) {\widetilde P}(\theta) \right) = 0$ can be shown in the same way.
Since $\rres$ is a trace,

\begin{eqnarray} \label{E:3.32}
& & \rres \left( i {\mathcal \gamma} T^{\prime}(\theta)  (sign {\mathcal A}) {\widetilde P}(\theta) \right) \hspace{0.1 cm} = \hspace{0.1 cm}
\rres \left( i {\widetilde P}(\theta) {\mathcal \gamma} T^{\prime}(\theta) (sign {\mathcal A}) {\widetilde P}(\theta) \right) \nonumber \\
& = & \rres \left( i {\mathcal \gamma} (I - {\widetilde P}(\theta))
T^{\prime}(\theta) (sign {\mathcal A}) {\widetilde P}(\theta) \right) \nonumber \\
& = & \rres  \left( i {\mathcal \gamma} \left( T^{\prime}(\theta) {\widetilde P}(\theta) +
\frac{i}{2} \left( \begin{array}{clcr} 0 & {\mathcal W}^{\ast} \\ - {\mathcal W} & 0 \end{array} \right) {\mathcal P}_{\ast} \right)
(sign {\mathcal A}) {\widetilde P}(\theta) \right) \nonumber \\
& = & \rres  \left( i {\mathcal \gamma} T^{\prime}(\theta) {\widetilde P}(\theta)
(sign {\mathcal A}) {\widetilde P}(\theta) \right) - \frac{1}{2} \rres
\left( {\mathcal \gamma} \left( \begin{array}{clcr} 0 & {\mathcal W}^{\ast} \\ - {\mathcal W} & 0 \end{array} \right) {\mathcal P}_{\ast}
(sign {\mathcal A}) {\widetilde P}(\theta) \right) \nonumber  \\
& = & cos\theta \rres  \left( i {\mathcal \gamma} T^{\prime}(\theta) {\widetilde P}(\theta) \right) - \frac{1}{2} \rres
\left( {\mathcal \gamma} \left( \begin{array}{clcr} 0 & {\mathcal W}^{\ast} \\ - {\mathcal W} & 0 \end{array} \right) {\mathcal P}_{\ast}
(sign {\mathcal A}) {\widetilde P}(\theta) \right). \nonumber
\end{eqnarray}

\noindent
We note that

\begin{eqnarray*}
\rres \left( i {\mathcal \gamma} T^{\prime}(\theta) {\widetilde P}(\theta) \right) =
\rres \left( i {\mathcal \gamma} T^{\prime}(\theta) ( I - {\widetilde P}(\theta)) \right) =
\rres \left( i {\mathcal \gamma} T^{\prime}(\theta) \right) -
\rres \left( i {\mathcal \gamma} T^{\prime}(\theta) {\widetilde P}(\theta) \right),
\end{eqnarray*}

\noindent
which together with Lemma \ref{Lemma:3.6} shows that

\begin{equation}   \label{E:3.34}
\rres \left( i {\mathcal \gamma} T^{\prime}(\theta) {\widetilde P}(\theta) \right) =
\frac{1}{2} \rres \left( i {\mathcal \gamma} T^{\prime}(\theta) \right) = 0.
\end{equation}

\vspace{0.2 cm}

We note that
$(sign {\mathcal A}) \hspace{0.1 cm} = \hspace{0.1 cm} \left( \begin{array}{clcr} 0 & U_{\Pi_{>}}^{\ast} \\ U_{\Pi_{>}} & 0 \end{array} \right)
{\mathcal P}_{\ast}$
and ${\mathcal \gamma}$ anticommutes with $\hspace{0.1 cm} (sign {\mathcal A}) \hspace{0.1 cm} $ and
 $ \hspace{0.1 cm} \left( \begin{array}{clcr} 0 & {\mathcal W}^{\ast} \\ - {\mathcal W} & 0 \end{array} \right) \hspace{0.1 cm}$.
Hence, we have

\begin{eqnarray*}
& & \rres \left( {\mathcal \gamma} \left( \begin{array}{clcr} 0 & {\mathcal W}^{\ast} \\ - {\mathcal W} & 0 \end{array} \right) {\mathcal P}_{\ast}
(sign {\mathcal A}) {\widetilde P}(\theta) \right)   \hspace{0.1 cm} = \hspace{0.1 cm}
\rres \left( {\mathcal \gamma} \left( \begin{array}{clcr} 0 & {\mathcal W}^{\ast} \\ - {\mathcal W} & 0 \end{array} \right) {\mathcal P}_{\ast}
(sign {\mathcal A}) ( I - {\widetilde P}(\theta)) \right),
\end{eqnarray*}

\noindent
which shows that

\begin{eqnarray}  \label{E:3.52}
& & \rres \left( {\mathcal \gamma} \left( \begin{array}{clcr} 0 & {\mathcal W}^{\ast} \\ - {\mathcal W} & 0 \end{array} \right) {\mathcal P}_{\ast}
(sign {\mathcal A}) {\widetilde P}(\theta) \right)   \nonumber \\
& = & \frac{1}{2} \rres \left( {\mathcal \gamma} \left( \begin{array}{clcr} 0 & {\mathcal W}^{\ast} \\ - {\mathcal W} & 0 \end{array} \right)
{\mathcal P}_{\ast} (sign {\mathcal A}) \right)
\hspace{0.1 cm} = \hspace{0.1 cm}
\frac{1}{2} \rres \left( \left( \begin{array}{clcr} i {\mathcal W}^{\ast} U_{\Pi_{>}} & 0 \\ 0 & i {\mathcal W} U_{\Pi_{>}}^{\ast} \end{array} \right)
{\mathcal P}_{\ast} \right)  \nonumber  \\
& = & - \frac{1}{2} \rres \left( {\mathcal \gamma} (sign {\mathcal A})
\left( \begin{array}{clcr} 0 & {\mathcal W}^{\ast} \\ - {\mathcal W} & 0 \end{array} \right) {\mathcal P}_{\ast} \right)
\hspace{0.1 cm} = \hspace{0.1 cm}
\frac{1}{2} \rres \left( \left( \begin{array}{clcr} i U_{\Pi_{>}}^{\ast} {\mathcal W}  & 0 \\ 0 & i U_{\Pi_{>}} {\mathcal W}^{\ast} \end{array} \right)
{\mathcal P}_{\ast} \right).
\end{eqnarray}

\noindent
The above equality with Lemma \ref{Lemma:3.1} and Lemma \ref{Lemma:3.7}  shows that

\begin{eqnarray}  \label{E:3.53}
\rres \left( {\mathcal \gamma} \left( \begin{array}{clcr} 0 & {\mathcal W}^{\ast} \\ - {\mathcal W} & 0 \end{array} \right) {\mathcal P}_{\ast}
(sign {\mathcal A}) {\widetilde P}(\theta) \right)
& = & \frac{i}{4} \rres \left( \left( \begin{array}{clcr} {\mathcal W}^{\ast} U_{\Pi_{>}} + U_{\Pi_{>}}^{\ast} {\mathcal W} & 0 \\
0 & {\mathcal W} U_{\Pi_{>}}^{\ast} + U_{\Pi_{>}} {\mathcal W}^{\ast} \end{array} \right) {\mathcal P}_{\ast} \right)  \nonumber  \\
& = & - \frac{i sin \theta}{2}  \rres \left( {\mathcal P}_{\ast}  \right) \hspace{0.1 cm} = \hspace{0.1 cm}  0,
\end{eqnarray}

\noindent
which completes the proof of the lemma.
\end{proof}

\vspace{0.2 cm}

For one parameter family of essentially self-adjoint Dirac operators $\B_{{\widetilde P}(\theta)}$ ($ 0 \leq \theta \leq \frac{\pi}{2}$)
we define the spectral flow $\SF (\B_{{\widetilde P}(\theta)})_{\theta \in [0, \frac{\pi}{2}]}$ by
$$
\SF (\B_{{\widetilde P}(\theta)})_{\theta \in [0, \frac{\pi}{2}]} \hspace{0.1 cm} := \hspace{0.1 cm} m^{+} - m^{-},
$$
where $m^{+}$ ($m^{-}$) is the number of eigenvalues which start negative (non-negative) and end non-negative (negative).
The following formula is well known (cf. Lemma 3.4 in [17]).

\vspace{0.2 cm}

\begin{equation}    \label{E:3.4557}
\eta( \B_{{\mathcal P}_{-, {\mathcal L}_{0}}} ) - \eta ( \B_{\Pi_{>, {\mathcal L}_{0}}} ) \hspace{0.1 cm} = \hspace{0.1 cm}
\SF (\B_{{\widetilde P}(\theta)})_{\theta \in [0, \frac{\pi}{2}]}  \hspace{0.1 cm} + \hspace{0.1 cm}
\int_{0}^{\frac{\pi}{2}} \frac{d}{d \theta} \eta_{\B_{{\widetilde P}(\theta)}}(0) \hspace{0.1 cm} d \theta.
\end{equation}

\vspace{0.2 cm}

\noindent
Lemma \ref{Lemma:3.9} and the result of Nicolaescu (Theorem 7.5 in [17], [23]) show that

\vspace{0.2 cm}

\begin{equation}    \label{E:3.4558}
\eta( \B_{{\mathcal P}_{-, {\mathcal L}_{0}}} ) - \eta ( \B_{\Pi_{>, {\mathcal L}_{0}}} ) \hspace{0.1 cm} = \hspace{0.1 cm}
\SF (\B_{{\widetilde P}(\theta)})_{\theta \in [0, \frac{\pi}{2}]} \hspace{0.1 cm} = \hspace{0.1 cm}
\Mas ( {\widetilde P}(\theta), \hspace{0.1 cm} {\mathcal C}_{M})_{\theta \in [0, \frac{\pi}{2}]} ,
\end{equation}

\vspace{0.2 cm}

\noindent
where ${\mathcal C}_{M}$ is the Calder\'on projector for $\B$ on $M$
and $\Mas ( {\widetilde P}(\theta), \hspace{0.1 cm} {\mathcal C}_{M})_{\theta \in [0, \frac{\pi}{2}]}$
is the Maslov index for the path ${\widetilde P}(\theta)$ and the constant path ${\mathcal C}_{M}$.
We refer to [17] and [23] for the definitions of the Maslov index and Calder\'on projector.

\vspace{0.2 cm}

The unitary operators corresponding to the projection ${\mathcal P}_{+}$ is $- \hspace{0.1 cm} U_{{\mathcal P}_{-}}$,
which shows that for   $0 \leq \theta \leq \frac{\pi}{2}$

$$
{\widetilde P}(- \theta) \hspace{0.1 cm} = \hspace{0.1 cm}
\frac{1}{2} \left( \begin{array}{clcr} \Id & P(- \theta)^{\ast} \\ P(- \theta) & \Id \end{array} \right) {\mathcal P}_{\ast}
 \hspace{0.1 cm} + \hspace{0.1 cm} {\mathcal P}_{{\mathcal L}_{0}}
$$

\vspace{0.2 cm}

\noindent
is a smooth path connecting $\Pi_{>, {\mathcal L}_{0}}$ and ${\mathcal P}_{+, {\mathcal L}_{0}}$.
Similar computation shows that

\begin{equation}    \label{E:3.4559}
\eta( \B_{{\mathcal P}_{+, {\mathcal L}_{0}}} ) - \eta ( \B_{\Pi_{>, {\mathcal L}_{0}}} ) \hspace{0.1 cm} = \hspace{0.1 cm}
\SF (\B_{{\widetilde P}(- \theta)})_{\theta \in [0, \frac{\pi}{2}]} \hspace{0.1 cm} = \hspace{0.1 cm}
\Mas ( {\widetilde P}(- \theta), \hspace{0.1 cm} {\mathcal C}_{M})_{\theta \in [0, \frac{\pi}{2}]} .
\end{equation}

\noindent
Summarizing the above arguments we have the following theorem, which is the main result of this section.

\vspace{0.2 cm}

\begin{theorem}   \label{Theorem:3.13}
Let $(M, g^{M})$ be a compact Riemannian manifold with boundary $Y$ and $g^{M}$ be a product metric near $Y$.
Then :

\vspace{0.2 cm}

\noindent
(1) $\hspace{0.2 cm} \eta( \B_{{\mathcal P}_{-, {\mathcal L}_{0}}} ) - \eta ( \B_{\Pi_{>, {\mathcal L}_{0}}} )
\hspace{0.1 cm} = \hspace{0.1 cm} \SF (\B_{{\widetilde P}(\theta)})_{\theta \in [0, \frac{\pi}{2}]}
\hspace{0.1 cm} = \hspace{0.1 cm}  \Mas ( {\widetilde P}(\theta), \hspace{0.1 cm} {\mathcal C}_{M})_{\theta \in [0, \frac{\pi}{2}]} $.  \newline
(2) $\hspace{0.2 cm} \eta( \B_{{\mathcal P}_{+, {\mathcal L}_{0}}} ) - \eta ( \B_{\Pi_{>, {\mathcal L}_{0}}} )
\hspace{0.1 cm} = \hspace{0.1 cm}  \SF (\B_{{\widetilde P}(- \theta)})_{\theta \in [0, \frac{\pi}{2}]}
\hspace{0.1 cm} = \hspace{0.1 cm} \Mas ( {\widetilde P}(- \theta), \hspace{0.1 cm} {\mathcal C}_{M})_{\theta \in [0, \frac{\pi}{2}]} $. \newline
\end{theorem}

\vspace{0.3 cm}

\section{Gluing formula of the refined analytic torsion}

\vspace{0.2 cm}

The gluing formula of the analytic torsion with respect to the relative and absolute boundary conditions ([9], [21], [29]) and the gluing formula of the eta invariant with respect to the APS boundary condition ([6], [7], [17], [32], [33]) are well known.
In this section we are going to use Theorem \ref{Theorem:2.11} and Theorem \ref{Theorem:3.13} together with results in [9], [6] and [17] to
obtain the gluing formula of the refined analytic torsion when $\nabla$ is an acyclic Hermitian connection.

Let $({\widehat M}, g^{\widehat M})$ be a closed Riemannian manifold of dimension $m = 2r -1$ and
${\widehat E} \rightarrow {\widehat M}$ be a flat vector bundle with a flat connection $\nabla$.
We denote by $Y$ a hypersurface of ${\widehat M}$ such that ${\widehat M}-Y$ has two components
whose closures are denoted by $M_{1}$ and $M_{2}$, {\it i.e.}
${\widehat M} = M_{1} \cup_{Y} M_{2}$.
We assume that $g^{\widehat M}$ is a product metric near $Y$ and that $\nabla$ is a Hermitian connection.
Let $\partial u$ be the unit normal vector field on a collar neighborhood of $Y$ such that $\partial u$ is outward on $M_{1}$ and inward on $M_{2}$.
We denote by $\B^{{\widehat M}}$ the odd signature operator on ${\widehat M}$ and
denote by $\B^{M_{1}}$, $\B^{M_{2}}$ ($E_{1}$, $E_{2}$, $g^{M_{1}}$, $g^{M_{2}}$) the restriction of
$\B^{{\widehat M}}$ (${\widehat E}$, $g^{{\widehat M}}$) to $M_{1}$, $M_{2}$.
We impose the boundary condition ${\mathcal P}_{+, {\mathcal L}_{1}}$ on $M_{1}$ and ${\mathcal P}_{-, {\mathcal L}_{0}}$ on $M_{2}$.
Then (\ref{E:1.17}) and (\ref{E:1.18}) show that

\begin{eqnarray} \label{E:4.1}
& & \log \Det_{\gr, \theta} (\B^{M_{1}}_{\even, {\mathcal P}_{+, {\mathcal L}_{1}}}) +
\log \Det_{\gr, \theta} (\B^{M_{2}}_{\even, {\mathcal P}_{-, {\mathcal L}_{0}}})   \\
& = & \frac{1}{2} \sum_{q=0}^{m} (-1)^{q+1} \cdot q \cdot \left( \log \Det_{2\theta} (\B^{M_{1}}_{q, {\widetilde {\mathcal P}_{1}}})^{2} +
\log \Det_{2\theta} (\B^{M_{2}}_{q, {\widetilde {\mathcal P}_{0}}})^{2}  \right)
 -  i \pi \hspace{0.1 cm} \left(  \eta(\B^{M_{1}}_{\even, {\mathcal P}_{+, {\mathcal L}_{1}}})  +
\hspace{0.1 cm} \eta(\B^{M_{2}}_{\even, {\mathcal P}_{-, {\mathcal L}_{0}}})  \right).   \nonumber
\end{eqnarray}

\vspace{0.2 cm}

\noindent
Theorem \ref{Theorem:2.11} together with Theorem 4.3 in [9] (p.36 in [9], cf. [21], [29]) leads to the following result.

\vspace{0.2 cm}

\begin{lemma}  \label{Lemma:4.1}
We assume that for each $0 \leq q \leq m$, $i = 1, 2$, $H^{q}({\widehat M}, {\widehat E}) = H^{q}(M_{i}, Y; E_{i}) = H^{q}(M_{i};E_{i}) = 0$. Then,

\begin{eqnarray*}
& & \frac{1}{2} \sum_{q=0}^{m} (-1)^{q+1} \cdot q \cdot \left( \log \Det_{2\theta} (\B^{M_{1}}_{q, {\widetilde {\mathcal P}_{1}}})^{2} +
\log \Det_{2\theta} (\B^{M_{2}}_{q, {\widetilde {\mathcal P}_{0}}})^{2}  \right)   \\
& = & \frac{1}{2} \sum_{q=0}^{m} (-1)^{q+1} \cdot q \cdot \left( \log \Det_{2\theta} (\B^{M_{1}}_{q, \Abs})^{2} +
\log \Det_{2\theta} (\B^{M_{2}}_{q, \rel})^{2}  \right)
 =  \frac{1}{2} \sum_{q=0}^{m} (-1)^{q+1} \cdot q \cdot \log \Det_{2\theta} (\B^{{\widehat M}}_{q})^{2}.
\end{eqnarray*}

\end{lemma}

\vspace{0.3 cm}

Under the assumption in Lemma \ref{Lemma:4.1}
Theorem \ref{Theorem:3.13} shows that
\begin{equation}  \label{E:4.2}
\eta ({\B^{M_{2}}_{{\mathcal P}_{-}}}) \hspace{0.1 cm} - \hspace{0.1 cm} \eta ({\B^{M_{2}}_{\Pi_{>}}})
 \hspace{0.1 cm} = \hspace{0.1 cm}   \Mas ( {\widetilde P}(\theta), \hspace{0.1 cm} {\mathcal C}_{M_{2}})_{\theta \in [0, \frac{\pi}{2}]}.
\end{equation}

\vspace{0.2 cm}

\noindent
We next consider $\eta ({\B^{M_{1}}_{{\mathcal P}_{+}}}) - \eta ({\B^{M_{1}}_{\Pi_{<}}})$.
Since $\partial u$ is the outward normal derivative on a collar neighborhood of $Y$ on $M_{1}$,
to use Theorem \ref{Theorem:3.13} we rewrite the odd signature operator $\B^{M_{1}}$ near the boundary by
$$
\B^{M_{1}} = {\mathcal \gamma} (\partial u + {\mathcal A}) = - {\mathcal \gamma} ( - \partial u - {\mathcal A}).
$$

\noindent
Here $- \partial u$ is the inward normal derivative to $Y$ on $M_{1}$.
Since $\Pi_{<}({\mathcal A}) = \Pi_{>}(- {\mathcal A})$ and $I - {\widetilde P}(\theta)$ is a path connecting
$\Pi_{<}({\mathcal A})$ and ${\mathcal P}_{+}$, Theorem \ref{Theorem:3.13} shows that

\begin{eqnarray}  \label{E:4.3}
\eta ({\B^{M_{1}}_{{\mathcal P}_{+}}}) \hspace{0.1 cm} - \hspace{0.1 cm} \eta ({\B^{M_{1}}_{\Pi_{<}({\mathcal A})}})
& = & \eta ({\B^{M_{1}}_{{\mathcal P}_{+}}}) \hspace{0.1 cm} - \hspace{0.1 cm} \eta ({\B^{M_{1}}_{\Pi_{>}(- {\mathcal A})}})
\hspace{0.1 cm} = \hspace{0.1 cm} \Mas ( I - {\widetilde P}(\theta), \hspace{0.1 cm} {\mathcal C}_{M_{1}})_{\theta \in [0, \frac{\pi}{2}]}.
\end{eqnarray}

\noindent
Equations (\ref{E:4.2}) and (\ref{E:4.3}) together with Theorem 8.8 in [17] show that

\begin{eqnarray}  \label{E:4.4}
& & \eta ({\B^{M_{1}}_{{\mathcal P}_{-}}}) \hspace{0.1 cm} + \hspace{0.1 cm} \eta ({\B^{M_{2}}_{{\mathcal P}_{+}}}) \nonumber \\
& = & \eta ({\B^{M_{1}}_{\Pi_{<}}}) \hspace{0.1 cm} + \hspace{0.1 cm} \eta ({\B^{M_{2}}_{\Pi_{>}}})
\hspace{0.1 cm} + \hspace{0.1 cm} \Mas ( {\widetilde P}(\theta), \hspace{0.1 cm} {\mathcal C}_{M_{2}})_{\theta \in [0, \frac{\pi}{2}]}
\hspace{0.1 cm} + \hspace{0.1 cm} \Mas ( I - {\widetilde P}(\theta), \hspace{0.1 cm} {\mathcal C}_{M_{1}})_{\theta \in [0, \frac{\pi}{2}]} \nonumber \\
& = & \eta ({\B^{{\widehat M}}})
\hspace{0.1 cm} + \hspace{0.1 cm} \Mas ( {\widetilde P}(\theta), \hspace{0.1 cm} {\mathcal C}_{M_{2}})_{\theta \in [0, \frac{\pi}{2}]}
\hspace{0.1 cm} + \hspace{0.1 cm} \Mas ( I - {\widetilde P}(\theta), \hspace{0.1 cm} {\mathcal C}_{M_{1}})_{\theta \in [0, \frac{\pi}{2}]} .
\end{eqnarray}

\vspace{0.2 cm}

\begin{lemma}  \label{Lemma:4.2}
Under the assumption of Lemma \ref{Lemma:4.1} we have :

\begin{eqnarray*}  \label{E:4.444}
\Mas ( {\widetilde P}(\theta), \hspace{0.1 cm} {\mathcal C}_{M_{2}})_{\theta \in [0, \frac{\pi}{2}]}
\hspace{0.1 cm} = \hspace{0.1 cm} \Mas ( I - {\widetilde P}(\theta), \hspace{0.1 cm} {\mathcal C}_{M_{1}})_{\theta \in [0, \frac{\pi}{2}]} .
\end{eqnarray*}

\noindent
In particular,
$\hspace{0.1 cm}  \eta ({\B^{M_{1}}_{{\mathcal P}_{-}}}) \hspace{0.1 cm} + \hspace{0.1 cm} \eta ({\B^{M_{2}}_{{\mathcal P}_{+}}})
\hspace{0.1 cm} \equiv \hspace{0.1 cm}
\eta ({\B^{{\widehat M}}}) \qquad (\Mod 2 \hspace{0.1 cm} {\Bbb Z}). $
\end{lemma}

\begin{proof}

We put $M_{1, r} = M_{1} \cup_{Y} \left( [0, r] \times Y \right)$, $M_{2, r} = M_{2} \cup_{Y} \left( [-r, 0] \times Y \right)$ and
 $M_{1, \infty} = M_{1} \cup_{Y} \left( [0, \infty) \times Y \right)$, $M_{2, \infty} = M_{2} \cup_{Y} \left( (- \infty, 0] \times Y \right)$.
We denote the extensions of $\B$ to $M_{i, r}$, $M_{i, \infty}$
by $\B_{M_{i, r}}$, $\B_{M_{i, \infty}}$ and
denote the corresponding Calder\'on projectors by ${\mathcal C}_{M_{i, r}}$, and
$\Imm {\mathcal C}_{M_{i, r}} := L_{M_{i, r}}$, and $\lim_{r \rightarrow \infty}  L_{M_{i, r}} := L_{M_{i, \infty}}$ for $i = 1, 2$.
We also denote the orthogonal projection to $L_{M_{i, \infty}}$ by ${\mathcal C}_{M_{i, \infty}}$.
Under the assumption of Lemma \ref{Lemma:4.1} it is shown in [17] (p.610 in [17]) that
$L_{M_{1, \infty}}$ and $L_{M_{2, \infty}}$ are Lagrangian subspaces
and $L_{M_{2, \infty}} = {\mathcal \gamma} L_{M_{1, \infty}}$.
Hence ${\mathcal C}_{M_{2, \infty}} = - {\mathcal \gamma} {\mathcal C}_{M_{1, \infty}} {\mathcal \gamma}$.
We define a homotopy $(F(\theta, s), \hspace{0.1 cm} G(\theta, s))$ on $M_{2}$ as follows.
$$
F(\theta, s) = {\widetilde P}(\theta), \qquad G(\theta, s) = {\mathcal C}_{M_{2, s}}, \qquad
(0 \leq \theta \leq \frac{\pi}{2}, \quad 0 \leq s \leq \infty).
$$
Then, $(F(\theta, 0), \hspace{0.1 cm} G(\theta, 0)) = ({\widetilde P}(\theta), \hspace{0.1 cm}  {\mathcal C}_{M_{2}})$ and
$(F(\theta, \infty), \hspace{0.1 cm} G(\theta, \infty)) = ({\widetilde P}(\theta), \hspace{0.1 cm}  {\mathcal C}_{M_{2, \infty}})$.
Since $\Ker \B_{\Pi_{>}}$ and $\Ker \B_{{\mathcal P}_{-}}$ are topological invariants (cf. Lemma \ref{Lemma:1.2} and Proposition 4.9 in [1]),
the assumption implies that

\begin{eqnarray*}
\Dim \left( \Ker F(0, s) \cap \Imm G(0, s) \right) & = &
\Dim \left( \Ker \Pi_{>}({\mathcal A}) \cap \Imm {\mathcal C}_{M_{2, s}} \right) \hspace{0.1 cm} = \hspace{0.1 cm}  0 ,  \\
\Dim \left( \Ker F(\frac{\pi}{2}, s) \cap \Imm G(\frac{\pi}{2}, s) \right) & = &
\Dim \left( \Ker {\mathcal P}_{-} \cap \Imm {\mathcal C}_{M_{2, s}} \right)
\hspace{0.1 cm} = \hspace{0.1 cm}  0 ,
\end{eqnarray*}

\noindent
which shows (cf. p.587 in [17]) that

$$
\Mas ( {\widetilde P}(\theta), \hspace{0.1 cm} {\mathcal C}_{M_{2}})_{\theta \in [0, \frac{\pi}{2}]}
\hspace{0.1 cm} = \hspace{0.1 cm} \Mas ( {\widetilde P}(\theta), \hspace{0.1 cm} {\mathcal C}_{M_{2, \infty}})_{\theta \in [0, \frac{\pi}{2}]}.
$$

\noindent
Similarly, we have

$$
\Mas ( I - {\widetilde P}(\theta), \hspace{0.1 cm} {\mathcal C}_{M_{1}})_{\theta \in [0, \frac{\pi}{2}]}
\hspace{0.1 cm} = \hspace{0.1 cm} \Mas ( I - {\widetilde P}(\theta), \hspace{0.1 cm} {\mathcal C}_{M_{1, \infty}})_{\theta \in [0, \frac{\pi}{2}]}
\hspace{0.1 cm} = \hspace{0.1 cm}
\Mas ( - {\mathcal \gamma} {\widetilde P}(\theta) {\mathcal \gamma},
\hspace{0.1 cm} - {\mathcal \gamma} {\mathcal C}_{M_{2, \infty}} {\mathcal \gamma} )_{\theta \in [0, \frac{\pi}{2}]}.
$$

\noindent
Hence, we have (cf. p.586 in [17])

\begin{eqnarray*}
\Mas ( {\widetilde P}(\theta), \hspace{0.1 cm} {\mathcal C}_{M_{2}})_{\theta \in [0, \frac{\pi}{2}]}
& = & \Mas ( {\widetilde P}(\theta), \hspace{0.1 cm} {\mathcal C}_{M_{2, \infty}})_{\theta \in [0, \frac{\pi}{2}]}
\hspace{0.1 cm} = \hspace{0.1 cm}
\Mas ( - {\mathcal \gamma} {\widetilde P}(\theta) {\mathcal \gamma},
\hspace{0.1 cm} - {\mathcal \gamma} {\mathcal C}_{M_{2, \infty}} {\mathcal \gamma})_{\theta \in [0, \frac{\pi}{2}]}  \\
& = &  \Mas ( I - {\widetilde P}(\theta), \hspace{0.1 cm} {\mathcal C}_{M_{1, \infty}})_{\theta \in [0, \frac{\pi}{2}]}
\hspace{0.1 cm} = \hspace{0.1 cm}
\Mas ( I - {\widetilde P}(\theta), \hspace{0.1 cm} {\mathcal C}_{M_{1}})_{\theta \in [0, \frac{\pi}{2}]} ,
\end{eqnarray*}

\noindent
which completes the proof of the lemma.
\end{proof}

\vspace{0.3 cm}

Under the assumption of Lemma \ref{Lemma:4.1} the refined analytic torsion $T_{{\widehat M}}(g^{{\widehat M}}, \nabla)$ on
${\widehat M}$ is defined by (Definition 10.1 in [4])

$$
\log T_{{\widehat M}}(g^{{\widehat M}}, \nabla) \hspace{0.1 cm} = \hspace{0.1 cm} \log \Det_{\gr, \theta} ( \B_{\even}^{{\widehat M}} )
\hspace{0.1 cm} + \hspace{0.1 cm} \pi i \hspace{0.1 cm} (\rank({\widehat E})) \hspace{0.1 cm}
\eta(\B_{\even}^{{\widehat M}, \trivial}).
$$

\vspace{0.2 cm}

\noindent
The refined analytic torsion $T_{M_{1}, {\mathcal P}_{+}}(g^{M_{1}}, \nabla)$ and $T_{M_{2}, {\mathcal P}_{-}}(g^{M_{2}}, \nabla)$ on $M_{1}$, $M_{2}$
with respect to the boundary conditions ${\mathcal P}_{+}$ and ${\mathcal P}_{-}$ are defined similarly (Dfinition 4.9 in [14]).
Lemma \ref{Lemma:4.1} and Lemma \ref{Lemma:4.2} lead to the following theorem, which is the main result of this paper.

\vspace{0.2 cm}

\begin{theorem}  \label{Theorem:4.3}
Let $({\widehat M}, g^{\widehat M})$ be a closed Riemannian manifold of dimension $m = 2r -1$ and $Y$ be a hypersurface so that
${\widehat M} = M_{1} \cup_{Y} M_{2}$. We assume that $g^{\widehat M}$ is a product metric near $Y$ and
for each $0 \leq q \leq m$, $i = 1, 2$, $H^{q}({\widehat M}, {\widehat E}) = H^{q}(M_{i}, Y; E_{i}) = H^{q}(M_{i};E_{i}) = 0$. Then,

\begin{eqnarray*}
\log T_{{\widehat M}}(g^{{\widehat M}}, \nabla) & = & \frac{1}{2} \sum_{q=0}^{m} (-1)^{q+1} \cdot q \cdot \log \Det_{2\theta} (\B^{{\widehat M}}_{q})^{2}
 -  i \pi \hspace{0.1 cm} \eta(\B^{{\widehat M}}_{\even}) + i \pi (\rank({\widehat E})) \eta(\B_{\even}^{{\widehat M}, \trivial})  \\
 & \equiv & \frac{1}{2} \sum_{q=0}^{m} (-1)^{q+1} \cdot q \cdot \left( \log \Det_{2\theta} (\B^{M_{1}}_{q, {\widetilde {\mathcal P}_{1}}})^{2} +
\log \Det_{2\theta} (\B^{M_{2}}_{q, {\widetilde {\mathcal P}_{0}}})^{2}  \right)  \\
& & -  i \pi \hspace{0.1 cm} \left( \eta(\B^{M_{1}}_{\even, {\mathcal P}_{+}}) + \eta(\B^{M_{2}}_{\even, {\mathcal P}_{-}}) \right)
+ i \pi (\rank({\widehat E}))
\hspace{0.1 cm} \left( \eta(\B^{M_{1}, \trivial}_{\even, {\mathcal P}_{+}}) + \eta(\B^{M_{2}, \trivial}_{\even, {\mathcal P}_{-}}) \right)  \\
& = & \log T_{M_{1}, {\mathcal P}_{+}}(g^{M_{1}}, \nabla) \hspace{0.1 cm} + \hspace{0.1 cm}
\log T_{M_{2}, {\mathcal P}_{-}}(g^{M_{2}}, \nabla) \qquad (\Mod \hspace{0.1 cm} 2\pi i {\Bbb Z}).
\end{eqnarray*}

\noindent
Equivalently, we have

$$
T_{{\widehat M}}(g^{{\widehat M}}, \nabla)  \hspace{0.1 cm} = \hspace{0.1 cm} T_{M_{1}, {\mathcal P}_{+}}(g^{M_{1}}, \nabla) \cdot
T_{M_{2}, {\mathcal P}_{-}}(g^{M_{2}}, \nabla).
$$

\end{theorem}

\vspace{0.3 cm}

%%%%%%%%%%%%%%%%%%%%%%%%%%%%%%%%%%%%%%%%%%%%%%%%%%%%%%%%%%%%%%%%%%%%%%%%%%%%%

\end{document}